\documentclass[11pt]{amsart}

 \usepackage[cmtip,all]{xy}

\usepackage{amsmath,amsthm,indentfirst}
\usepackage{amssymb}
\usepackage{amsfonts}
\usepackage{amscd}
\usepackage[latin1]{inputenc}
\usepackage{hyperref}
\usepackage{ifthen, amsfonts, amssymb, graphicx, srcltx, mathrsfs, xfrac,
amsmath}
\usepackage{enumerate}
\usepackage{pinlabel}
\input xy 
\xyoption{all}
\usepackage{xcolor}
 \usepackage{xspace,xcolor}
 \usepackage{graphicx}
 \usepackage{tikz}
 
\usepackage{soul}

\theoremstyle{plain}
\newtheorem*{maintheorem*}{Main Theorem}

\newtheorem*{thmd*}{Theorem 1.5}
\newtheorem*{thme*}{Theorem 1.6}
\newtheorem*{remark*}{Remark}
\newtheorem*{conjecture*}{Conjecture}
\newtheorem*{prop*}{Proposition}
\newtheorem{thm}{Theorem}[section]
\newtheorem{cor}[thm]{Corollary}
\newtheorem{lem}[thm]{Lemma}
\newtheorem{prop}[thm]{Proposition}

\theoremstyle{definition}

\newtheorem*{proofc*}{Proof of Theorem C}

\newtheorem{remark}[thm]{Remark}
\newtheorem{notation}[thm]{Notation}

\DeclareMathOperator{\CAT}{CAT}
\DeclareMathOperator{\Teich}{Teich}
\DeclareMathOperator{\Mod}{Mod}

\DeclareMathOperator{\WP}{WP}
\DeclareMathOperator{\base}{base}
\DeclareMathOperator{\tw}{tw}

\DeclareMathOperator{\inj}{inj}

\DeclareMathOperator{\diam}{diam}

  \newcommand{\asym}{\stackrel{{}_\ast}{\asymp}}
  \newcommand{\gmul}{\stackrel{{}_\ast}{\succ}}
  \newcommand{\lmul}{\stackrel{{}_\ast}{\prec}}
  \newcommand{\asya}{\stackrel{{}_+}{\asymp}}
  
  \newcommand{\ladd}{\stackrel{{}_+}{\prec}}

\newcommand{\ML}{\mathcal{ML}}
\newcommand{\PML}{\mathcal{PML}}

\newcommand{\cC}{\mathcal C}

\newcommand{\cS}{\mathcal S}
\newcommand{\cT}{\mathcal T}

\newcommand{\I}{\text{i}}
 
\newcommand{\ep}{\epsilon} 
\newcommand{\T}{Teichm\"{u}ller} 
\newcommand{\oT}[1]{\overline{\Teich(#1)}}

\numberwithin{equation}{section}

\begin{document}


\title[Limit sets of WP geodesics]{Limit sets of Weil-Petersson geodesics with nonminimal ending laminations}

\date{\today}


\author[Jeff Brock]{Jeffrey Brock}
\address{Department of Mathematics, Brown University, Providence, RI, }
\email{brock@math.brown.edu}

\author[Chris Leininger]{Christopher Leininger}
\address{ Department of Mathematics, University of Illinois, 1409 W Green ST, Urbana, IL}
\email{clein@math.uiuc.edu}

\author[Babak Modami]{Babak Modami}
\address{ Department of Mathematics, Yale University, 10 Hillhouse Ave, New Haven, CT}
\email{babak.modami@yale.edu}

\author[Kasra Rafi]{Kasra Rafi}
\address{Department of Mathematics, University of Toronto, Toronto, ON }
\email{rafi@math.toronto.edu}

\thanks{The first author was partially supported by NSF grant DMS-1207572, the second author by NSF grant DMS-1510034, the third by NSF grant DMS-1065872, and the fourth
author by NSERC grant \# 435885.}




\begin{abstract}
In this paper we construct examples of Weil-Petersson geodesics with nonminimal ending laminations which have $1$--dimensional limit sets in the Thurston compactification of \T\ space. 

\end{abstract}

\maketitle


\section{Introduction}

A number of authors have studied the limiting behavior of Teichm\"uller geodesics in relation to the Thurston compactification of Teichm\"uller space, \cite{2bdriesteich,Kerck-asymp-teich} \cite{lenzhen,lenmasdivergence},\cite{nonuniqueerg,nuechaikaetl},\cite{nue2,2dlimit}.  This work has highlighted the delicate relationship between the vertical foliation of the quadratic differential defining the geodesic and the limit set in the Thurston boundary.

The {\it ending lamination} of a Weil-Petersson (WP) geodesic ray was introduced by Brock, Masur and Minsky 
in \cite{bmm1} and in some sense serves as a rough analogue of the vertical foliation of the quadratic differential defining a Teichm\"uller geodesic ray.
Ending laminations have been used to study the behavior of WP geodesics \cite{bmm1,bmm2,wpbehavior, asympdiv, wplimit} 
and dynamics of the WP geodesic flow on moduli spaces \cite{bmm2,wprecurnue,Hamen-teich-wp}. 
In this paper, complementing our work in \cite{wplimit}, we provide examples of WP geodesic rays with non minimal, 
and hence nonuniquely ergodic, ending laminations whose limit sets in the Thurston compactification of \T\ space is larger than a single point.
\medskip

\noindent{\bf Theorem~\ref{thm:main}.} {\em There exist Weil-Petersson geodesic rays with nonminimal, nonuniquely ergodic ending laminations whose limit set 
in the Thurston compactification of \T\ space is $1$--dimensional.}
\medskip

See also Theorem~\ref{thm : limit of r} for a more precise statement.  Our construction closely follows that of Lenzhen \cite{lenzhen} who gave the first examples of Teichm\"uller geodesics having $1$--dimensional limit sets in the Thurston compactification.

\pagebreak

\section{Preliminaries}\label{sec:prelim}

\begin{notation}
Let $K\geq 1$, $C\geq 0$, and let $X$ be any set.  
For two functions $f,g:X\to[0,\infty)$ we write $f\asymp_{K,C}g$ if $\frac{1}{K}g(x)-C\leq f(x)\leq Kg(x)+C$ for all $x\in X$. 
Similarly, we write $f\asym_K g$ if $\frac{1}{K}g(x)\leq f(x)\leq Kg(x)$ for all $x\in X$, and $f\asya_C g$ if $g(x)-C\leq f(x)\leq g(x)+C$ for all $x\in X$. 
Moreover, $f\lmul_K g$ means that $f(x)\leq Kg(x)$ for all $x\in X$ and $f\ladd_C g$ means that $f(x)\leq g(x)+C$ for all $x\in X$.
We drop $K,C$ from the notation when the constants are understood from the context. 
\end{notation}

\noindent{\bf \T\ space.} Given a finite type surface $S$, we denote its Teichm\"uller space by $\Teich(S)$.  The points in $\Teich(S)$ are isotopy classes of (finite type) Riemann surface structures on $S$.  When the Euler characteristic $\chi(S) < 0$, we also view $X \in \Teich(S)$ as an isotopy class of complete, finite area, hyperbolic metric on $S$.  In this case, given a homotopy class of closed curve $\alpha$ and $X \in \Teich(S)$, we write $\ell_\alpha(X)$ for the length of the $X$--geodesic representative of $\alpha$.  If $\alpha$ is simple, we let $w_\alpha(X)$ denote the {\em width} of $\alpha$ in $X$, defined by
\begin{equation} \label{Eq:collar width}
w_\alpha(X) = 2\sinh^{-1} ( 1/ \sinh(\ell_\alpha(X)/2)).
\end{equation}
The term `width' is justified by the following, see e.g.~\cite[\S 4]{buser}.
\begin{lem}[Collar Lemma] \label{L:collar}  Given any $X \in \Teich(S)$ and distinct homotopy classes of disjoint simple closed curves $\alpha_1,\alpha_2$, let $w_i = w_{\alpha_i}(X)$, for $i=1,2$.  Then $N_{w_1/2}(\alpha_1)$ and $N_{w_2/2}(\alpha_2)$, the $w_i/2$--neighborhoods of the $X$--geodesic representative of the $\alpha_i$, are pairwise disjoint, embedded annuli.
\end{lem}
For $w = w_\alpha(X)$, we call $N_{w/2}(\alpha)$, the {\em standard collar}, and note that the distance inside $N_{w/2}(\alpha)$ between the boundary components is $w$.  An important consequence is that for any other homotopy class of curve $\beta$, we have $\ell_\beta(X) \geq i(\alpha,\beta) w_\alpha(X)$, where $i(\alpha,\beta)$ is the geometric intersection number of $\alpha$ and $\beta$ (c.f.~Theorem~\ref{thm:combinatorial length} below).
\medskip

\noindent{\bf Weil-Petersson metric.} When $\chi(S) < 0$, the Weil-Petersson (WP) metric is a negatively curved, incomplete, geodesically convex, Riemannian metric on $\Teich(S)$.
Its completion, $\oT{S}$, is a stratified $\CAT(0)$ space, with a stratum $\cS(\sigma)$ for each (possibly empty isotopy class of) multicurve $\sigma$, consisting of appropriately marked Riemann surfaces pinched precisely along $\sigma$.
The stratum $\cS(\sigma)$ is totally geodesic and isometric to the product of the \T\ spaces of the connected components of $S\backslash \sigma$ with their WP metric.  
The completion of $\cS(\sigma)$ is the union of all strata $\cS(\sigma')$ for which $\sigma \subset \sigma'$; see \cite{maswp}.   
The stratification has the so called {\it non-refraction property}:  the interior of a geodesic segment with end points in two strata $\cS(\sigma_1)$ and $\cS(\sigma_2)$ lies in the stratum $\cS(\sigma_1\cap\sigma_2)$; see \cite{dwwp,wolb}.\\


\noindent{\bf Curve complexes, markings, and projections.}
We refer the reader to \cite{mm1,mm2} for definitions of the objects described in this subsection---our objective here is to fix notation and terminology.
 In this paper we denote the {\em curve complex} of a subsurface $Y$ by $\cC(Y)$. 
 The set of vertices of $\cC(Y)$, denoted by $\cC_0(Y)$, is the set of curves on $Y$ (more precisely, the set of isotopy classes of essential simple closed curves on $Y$). 
 A {\em partial marking} $\mu$ on $S$ consists of a pants decomposition, $\base(\mu)$, and a transversal for some curves in $\base(\mu)$.  A {\em marking} is a partial marking such that every curve in $\base(\mu)$ has a transversal.  
 For a curve or partial marking $\mu$, we denote the {\em subsurface projection of $\mu$} to the subsurface $Y$ by $\pi_Y(\mu)$ (see \cite[\S 2]{mm2}), and for two $\mu,\mu'$ define
\begin{equation}
d_Y(\mu,\mu'):=\diam_{\cC(Y)}\Big(\pi_Y(\mu)\cup\pi_Y(\mu')\Big).
\end{equation}
An important property of $d_Y$ is that it satisfies the triangle inequality when the associated projections are nonempty.
If $Y$ is an annulus with core curve $\alpha$, we also write $\cC(\alpha)$ for $\cC(Y)$, $\pi_\alpha$ for $\pi_Y$, and
$d_\alpha(\mu,\mu')$ for $d_Y(\mu,\mu')$; see again \cite[\S 2]{mm2}. 

There exists a constant $L_S > 0$, called the {\em Bers constant}, depending on $S$, 
such that for any $X \in \Teich(S)$ there is a pants decomposition such that every curve in the pants decomposition
 has hyperbolic length at most $L_S$ with respect to $X$; see e.g.~\cite{buser}.  Such a pants decomposition is called a {\em Bers pants decomposition} for $X$.  
A {\em Bers curve} for $X$ is a curve $\alpha$ for which $\ell_\alpha(X)\leq L_S$.  
A  {\em Bers marking} for $X$ is a marking $\mu$ such that $\base(\mu)$ is a Bers pants decomposition for $X$
and transversal curves have minimal lengths.

Given a point $X\in\Teich(S)$ and a curve $\alpha$, the {\em subsurface projection of $X$} to $\alpha$, $\pi_\alpha(X)$, is the collection of all geodesic arcs in the annular cover corresponding to $\alpha$ which are orthogonal to the geodesic representative of $\alpha$ (all with respect to the pull-back of the $X$--metric on $S$ to the cover).  Distance in $\alpha$ between points of $\Teich(S)$ and curves/markings is defined as the diameter of the union of their projections (as with the case of two curves or markings).  This is often called the {\em relative twisting}, and for $\alpha,\delta \in \cC_0(S)$ and $X \in \Teich(S)$, we write
\[ \tw_\alpha(\delta,X) = d_\alpha(\delta,X) := d_\alpha(\delta,\pi_\alpha(X)).\]
If $\alpha$ has bounded length and $\mu$ is a bounded length marking for $X$, then 
\begin{equation}\label{eq:twist-marking}
 \tw_\alpha(\delta,X) \asya d_\alpha(\delta,\mu),
\end{equation}
where the additive error depends on the bounds on the length of $\alpha$ and the lengths of those curves in $\mu$ (including those defining transversals of $\mu$) which intersect $\alpha$, but not on the length of $\delta$.  To see this, note that the bounds on all the lengths of curves mentioned implies a lower bound on the length of $\alpha$ by Lemma~\ref{L:collar} and a lower bound on the angle of intersection between the geodesic representatives of any curve from $\mu$ and the geodesic representative of $\alpha$, and these easily imply an upper bound $d_\alpha(\mu,X)$.
Coarse Equation (\ref{eq:twist-marking}) then follows from the triangle inequality.

The next theorem is a consequence of \cite[Lemma 3.1]{Shadow} 
(see also \cite[Lemmas 7.2 and 7.3]{lineminimateichgeod}), and provides an estimate 
on length of a curve $\gamma$ with respect to $X \in \Teich(S)$ in terms of contributions 
from certain other curves which $\gamma$ intersects.  To describe it, suppose 
$X \in \Teich(S)$ and $\gamma,\delta$ are two curves on $S$, and define
\begin{equation} \label{Eq:contribution delta gamma X}
\ell_\delta(\gamma,X)=\I(\delta,\gamma)
\Big(w_{\gamma}(X)+\ell_\gamma(X)\tw_\gamma(\delta,X)\Big).
\end{equation}
Also, for a pants decomposition $P$, define 
\[
\I(\delta, P) = \sum_{\gamma\in P} \I(\delta,\gamma).
\]

\begin{thm}\label{thm:combinatorial length}
For any $L>0$ there exists $K>0$ so that the following holds. Let 
$X \in \Teich(S)$ and $P$ is a pants decomposition of $S$ with $\ell_\gamma(X)\leq L$
for all $\gamma\in P$. Then for any curve $\delta \in \cC_0(S)$, $\delta \not \in P$,
we have 
\[
\Big| \ell_\delta(X)-\sum_{\gamma\in P} \ell_\delta(\gamma,X)\Big|
\lmul_K \I(\delta,P)
\]
\end{thm}
\begin{proof}
For every $\gamma \in P$, let $N(\gamma) = N_{w_\gamma(X)/2}(\gamma)$ be the standard collar around $\gamma$, 
 where $w_\gamma(X)$ is the width as in (\ref{Eq:collar width}).
By Lemma~\ref{L:collar}, these collars are embedded and pairwise disjoint. 
Every complementary component $Q$ of this set of standard collars is topologically 
a pair of pants but does not have a geodesic boundary. The decomposition of $X$
into standard collars and complementary components decomposes $\delta$ into 
segments.  Then \cite[Lemma 3.1, part (b)]{Shadow} implies that for any segment $u$ 
that is associated to a standard collar $N(\gamma)$ we have, 
\[
\ell_u(X) \asya_C w_{\gamma}(X)+\ell_\gamma(X)\tw_\gamma(\delta,X)
\]
for some constant $C$ depending on $L$. (The language in 
\cite[Lemma 3.1]{Shadow} is slightly different because it also applies to segments in
possibly infinite geodesics.) That is, for some constant $K_1$, we have 
\[
\Big| \sum_u \ell_u(X) -\sum_{\gamma\in P} \ell_\delta(\gamma,X)\Big|
\lmul_{K_1} \I(\delta,P). 
\]
But the difference between $\ell_\delta(X)$ and $\sum_u \ell_u(X) $
is the sum of the lengths of segments in complementary pieces. 
Now we note that \cite[Lemma 3.1, part (a)]{Shadow} states that the length of each such segment 
is uniformly bounded. Also, the number of such segments is 
$\I(\delta,P)$. Thus, for some $K_2$, 
\[
\Big| \ell_\delta(X) - \sum_u \ell_u(X) \Big| 
\lmul_{K_2} \I(\delta,P). 
\]
Now, setting $K=K_1+K_2$, the theorem follows from above two inequalities and
the triangle inequality. 
  \end{proof}
\medskip

\noindent{\bf The Thurston compactification.}
The {\em Thurston boundary of the \T\ space} is the space of projective classes of measured laminations $\mathcal{PML}(S)$; 
see \cite{FLP}. A sequence of points $\{X_k\} \subset \Teich(S)$ exiting every compact set of $\Teich(S)$ converges to $[\bar{\lambda}]$, the projective class of a measured lamination $\bar{\lambda} \in \ML(S)$, 
if there exists a sequence of positive real numbers $\{u_k\}_k$ so that
\begin{equation} \label{EQ:Thurston limit definition}
\lim_{k\to\infty}u_k\ell_{\delta}(X_k)=\I(\delta,\bar\lambda),
\end{equation}
for every $\delta \in \cC_0(S)$.  
We call $\{u_k\}_k$ a {\em scaling sequence} for $\{X_k\}_k$, and note that $u_k \to 0$.
In fact, a finite set of curves $\delta_1,\ldots,\delta_n$ can be chosen so that for any sequence $\{X_k\}$ exiting every compact subset of $\Teich(S)$, we have
$X_k \to [\bar \lambda]$ if and only if (\ref{EQ:Thurston limit definition}) holds for some scaling sequence $\{u_k\}_k$ and the curves $\delta = \delta_i$, for each $i = 1,\ldots,n$.
To see this, we let $\delta_1,\ldots,\delta_n$ consist of a pants decomposition together with a pair of transverse curves for each pants curve.  Then any measured foliation/lamination is determined by these intersection numbers (indeed, the intersection numbers with the transverse curves suffice to determine the twisting parameters with for the foliation, and hence the foliation; see \cite[Expos\'e 6]{FLP}).  Therefore, if (\ref{EQ:Thurston limit definition}) holds for some $\bar \lambda$, some $\{u_k\}$, and $\delta = \delta_i$, for all $i = 1,\ldots,n$, then all accumulation points of $\{X_k\}$ agree (as they are determined by these intersection numbers), and hence $\{X_k\}$ converges to $[\bar \lambda]$.  In particular, for any curve $\alpha$, we can choose the curves $\delta_1,\ldots,\delta_n$ to all have nonzero intersection number with $\alpha$.
\medskip

\noindent{\bf Ending lamination.} Suppose $r \colon [0,\infty) \to \Teich(S)$ is an infinite WP geodesic ray.  A {\em pinching curve} for $r$  is a curve $\gamma$ with $\lim_{t \to \infty} \ell_\gamma(r(t))=0$.  The {\em (forward) ending lamination of $r$}, denoted by $\nu^+ = \nu^+(r)$, is the union of the pinching curves together with the supports of any accumulation points in $\PML(S)$ of an infinite sequence of distinct Bers curves for hyperbolic metrics along $r([0,\infty)$; see \cite[Definition 2.7]{bmm1} for more details.
 
 \subsection{Bounded length WP geodesic segments}
Because of the non-completeness of the Weil-Petersson metric and the non-local-compactness of its completion, the
usual compactness theorems for geodesic segments of fixed length
based at a point is more subtle than in the complete case. Wolpert
carried out an initial analysis \cite[Proposition 23]{wols} that captured how such
segments can limit at the completion, but further analysis in
\cite[Theorem 4.2]{wpbehavior} captures a stronger non-refraction condition.

Given a curve $\gamma\in \cC(S)$ we denote the positive Dehn twist
about $\gamma$ by $D_{\gamma}$.  
For a multicurve $\sigma$ on a surface $S$ we denote the subgroup of
$\Mod(S)$ generated by positive Dehn twists about the curves in
$\sigma$ by $\tw(\sigma)$.


\begin{thm}\textnormal{(Geodesic limit)}\label{thm : geodesic limit}
Given $T>0$, let  $\zeta_{n}:[0,T]\to \Teich(S)$ be a sequence of
Weil-Petersson  geodesic segments parametrized by arclength with
$\zeta_n(0) = X \in \Teich(S)$. Then
after passing to a subsequence, we may extract 
 a partition of the interval $[0,T]$ by $0=t_0<t_1<\ldots<t_{k+1}=T$,
 multicurves $\sigma_l, \; l=1,\ldots,k+1$, 
with 
$\sigma_l\cap\sigma_{l+1}=\emptyset$ for $l=1,\ldots, k$
 and a piecewise geodesic segment 
\[\hat{\zeta}:[0,T]\to\oT{S}\]
with  $\hat{\zeta}(t_l)\in\cS(\sigma_l)$ for $ l=1,\ldots,k+1,$
and $\hat{\zeta}((t_{l},t_{l+1}))\subset\Teich(S)$ for $ l=0,\ldots,k,$ 
such that the following hold:
\begin{enumerate}
  \item 
$\lim_{n\to\infty}\zeta_n(t)=\hat{\zeta}(t)$
 for all $t\in[t_0,t_1]$,
\item  
there exist elements $\mathcal{T}_{l,n}\in\tw(\sigma_l)$ 
for each $l=1,\ldots, k$ and $n\in \mathbb{N}$, so that letting
$\varphi_{l,n}=\mathcal{T}_{l,n}\circ \ldots\circ \mathcal{T}_{1,n}$
we have
\[\lim_{n\to\infty}\varphi_{l,n}(\zeta_{n}(t))=\hat{\zeta}(t)\] 
 for all $t\in [t_l,t_{l+1}]$.
 \end{enumerate}
\end{thm}

\begin{remark}
In this theorem, $\sigma_{k+1}$ may be empty (in which case we have $\hat \zeta(t_{k+1}) \in \Teich(S)$).
A key feature of this theorem is that $\sigma_l \cap \sigma_{l+1} = \emptyset$, meaning that these two multicurves have no common components.  This is responsible for the non-refraction behavior ensuring that $\hat{\zeta}((t_l,t_{l+1}))$ is contained in $\Teich(S)$ as opposed to $\oT{S}$.
\end{remark}

We also need the following, which is \cite[Corollary 4.10]{wpbehavior}.
Denote a Bers marking at a point $X\in\Teich(S)$ by $\mu(X)$.

\begin{thm}\label{thm : tw-sh} 
Given $\ep_0,T$ positive and $\ep\in (0,\ep_0]$, there is an $N\in\mathbb N$ with the following property.  
Suppose that $\zeta \colon [a,b] \to \Teich(S)$ is a WP geodesic segment of length at most $T$ such that $\sup_{t \in [a,b]} \ell_\alpha(\zeta(t)) \geq \ep_0$ and $d_\alpha(\mu(\zeta(a)),\mu(\zeta(b))) > N$.
Then, we have
\[ \inf_{t \in [a,b]} \ell_\alpha(\zeta(t)) \leq \ep.\]
\end{thm}

\section{Geodesics with nonminimal ending laminations} \label{sec : nonminimal}

In this section we prove the main result of the paper (see also Theorem~\ref{thm : limit of r} for a more precise statement).

\begin{thm}\label{thm:main}
There exist Weil-Petersson geodesic rays with nonminimal, nonuniquely ergodic ending laminations whose limit set 
in the Thurston compactification of \T\ space is $1$--dimensional.
\end{thm}

First, let us briefly sketch our construction of such geodesic rays.  
The basic idea is similar to Lenzhen's construction for Teichm\"uller geodesics in \cite{lenzhen}.   
Let $S$ be the closed, genus $2$ surface and let $\alpha \subset S$ be a separating simple closed curve cutting $S$ into two one-holed tori that we denote by $S_0$ and $S_1$.  The stratum $\cS(\alpha)$ is isometric to a product of \T\ spaces of once-punctured tori, i.e., $\cS(\alpha) \cong \Teich(S_0) \times \Teich(S_1)$.

We carefully choose sequences of curves $\{\gamma^h_i\}_i\subset \cC(S_h),\; h=0,1$ which form quasi-geodesics and limit to minimal filling laminations $\lambda_h,\; h=0,1$.  Using the fact that $\Teich(S_h)$ with the WP metric is quasi-isometric to $\cC(S_h)$, and that it has negative curvature bounded away from $0$ we construct geodesic rays $\hat r^h$ in $\Teich(S_h)$ which have forward ending laminations $\lambda_h$, $h = 0,1$. 

Next, we consider the geodesic $\hat r=(\hat r^0,\hat r^1)$ in $\oT{S}$, and construct a geodesic ray $r$ which fellow travels $\hat r$.  
We estimate the length of an arbitrary curve along $r$ using estimates from Theorem \ref{thm:combinatorial length}.  
From the conditions we imposed on our sequences of curves, 
we will see that most of the length of the curve comes from its intersection with curves $\gamma^0_i$ and $\gamma^1_i$, and so lengths are eventually well--approximated by intersection numbers with linear combinations of measure $\bar \lambda_0$ and $\bar \lambda_1$ on $\lambda_0$ and $\lambda_1$, respectively.
Consequently, this geodesic ray accumulates on a $1$--simplex with vertices $[\bar\lambda_0]$ and $[\bar\lambda_1]$ in the Thurston boundary.  
Analyzing a pair of particular sequences of times, we see that the endpoints of the simplex are in the limit set, and so by connectivity, the limit set consists of the entire $1$--simplex.

\subsection{Continued fraction expansions and geodesics in $\Teich(S_{1,1})$}\label{subsec : contfrac}


Let $\lambda_h$ be a minimal, irrational lamination on $S_{h}$.  
This lamination is the straightening of a foliation of the flat square torus, and we assume for convenience that the slope of the leaves of this foliation is greater than $1$.  
The reciprocal of this slope is an irrational number less than $1$ which we denote by $x_h$, and we write its continued fraction expansion as
 \begin{eqnarray}
x_h=[0 ; e^h_0,e^h_1,\ldots],
 \end{eqnarray}
(the first coefficient is zero since  $x_h  < 1$).  We assume in all that follows that $e_i^h \geq 4$ for all $i$ and for $h=0,1$.

\begin{figure}
\centering
\scalebox{0.2}{}
 \begin{tikzpicture}[scale=12]
  
    \draw [line width=0.01cm] (0,0) -- (1,0);
    \draw [line width=0.01cm] (0,0) -- (0,.5);
    \draw [line width=0.01cm] (1,0) -- (1,.5);
    \draw [dotted] (0,0) -- (0,-0.05) node[below]{$\frac{0}{1}$};
    \draw [dotted] (1,0) -- (1,-0.05) node[below]{$\frac{1}{1}$};
       
     \draw [line width=0.01cm] (1/2,0) arc (0:180:1/12);
    \draw [dotted] (1/3,0) -- (1/3,-0.05) node[below]{$\frac{1}{3}$};
    
     \draw [line width=0.01cm] (1/3,0) arc (0:180:1/6);
       \draw [line width=0.01cm] (1,0) arc (0:180:1/2);

        \draw [line width=0.01cm] (1,0) arc (0:180:1/4);
    \draw [dotted] (.5,0) -- (1/2,-0.05) node[below]{$\frac{1}{2}$};
    
    \draw [line width=0.01 cm] (1,0) arc (0:180:1/6);
    \draw [dotted] (2/3,0) -- (2/3,-0.05) node[below]{$\frac{2}{3}$};
    
      \draw [line width=0.01cm] (1,0) arc (0:180:1/8);
\draw  [line width=0.01cm] (1/2,0) arc (0:180:1/4);
\draw [line width=0.01cm] (2/3,0) arc (0:180:1/12);

\draw [line width=.01cm] (3/4,0) arc (0:180:1/24);
\draw [dotted] (3/4,0) -- (3/4,-0.05) node[below]{$\frac{3}{4}$};

\draw [line width=.01 cm] (1/4,0) arc (0:180:1/8);
\draw [line width=.01 cm] (1/3,0) arc (0:180:1/24);
\draw [dotted] (1/4,0) -- (1/4,-.05) node[below]{$\frac14$};

\draw [line width=.01 cm] (2/5,0) arc (0:180:1/30);
\draw [line width=.01 cm] (1/2,0) arc (0:180:1/20);
\draw [dotted] (2/5,0) -- (2/5,-.05) node[below]{$\frac25$};

\draw [line width=.01 cm] (3/5,0) arc (0:180:1/20);
\draw [line width=.01 cm] (2/3,0) arc (0:180:1/30);
\draw [dotted] (3/5,0) -- (3/5,-.05) node[below]{$\frac35$};

\end{tikzpicture}
     
  \caption{A piece of the curve graph $\cC(S_h)$, visualized as the Farey graph.}
\label{fig:farey}
  \end{figure}
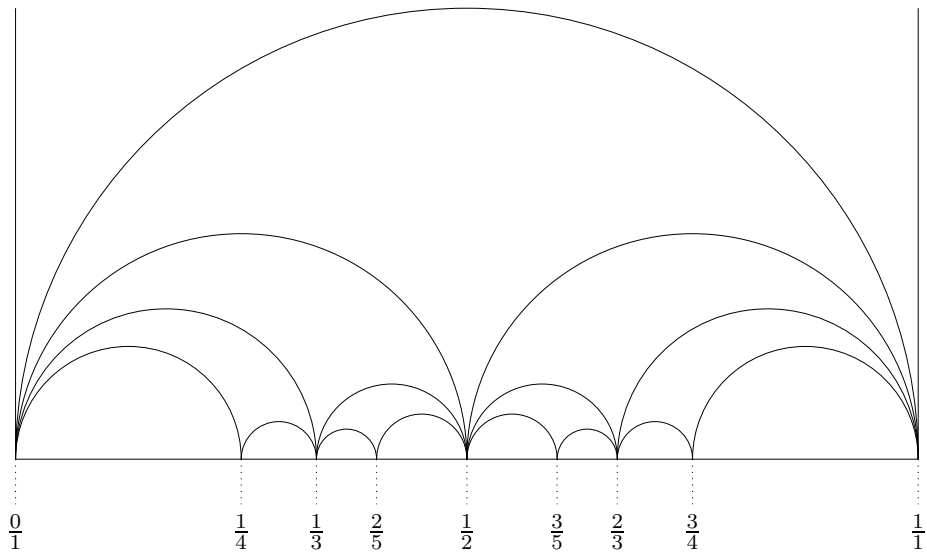

Next, let $\frac{p_i^h}{q_i^h} = [0;e^h_0,e^h_1,\ldots,e^h_{i-1}]$ be the {\em $i^{th}$ convergent} with finite continued fraction expansion as shown, obtained by truncating that of $x_h$.  
Let $\gamma_i^h$ be the simple closed curve on the torus whose slope is the reciprocal, $\frac{q_i^h}{p_i^h}$.  
Note that $\gamma_0^h$ is the curve whose reciprocal slope is $0$ (that is, $\gamma_0^h$ is the vertical curve) and we let $\gamma_{-1}^h$ denote the horizontal curve, by convention.

The {\em Farey graph} is the graph with vertices corresponding to $\mathbb Q \cup \{\infty\}$ and edges between $\frac{p}{q}$ and $\frac{r}{s}$ whenever $|ps-rq|=1$ \cite{modsurf}; see Figure~\ref{fig:farey}.
Identifying a simple closed curve on the (flat, square) torus with the reciprocal of its slope identifies the curve graph $\cC(S_h)$ with Farey graph \cite{geometricccplx}, 
and we use these two graphs interchangeably depending on our purposes. Our assumption that $e_i^h \geq 4$ ensures that the sequence of curves $\{\gamma_i^h\}_i$ 
(or equivalently, the sequence of convergents $\{\frac{p_i^h}{q_i^h}\}_i$) is a geodesic; see e.g.~\cite[Section 3]{geometricccplx}.  Our index convention leads to
\begin{equation}\label{eq : tw gh i-1, i+1}
\gamma^h_{i+1}=D_{\gamma^h_i}^{\pm e^h_i}(\gamma^h_{i-1}),
\end{equation}
with the sign determined by the parity of $i$; see Figure \ref{fig:farey2}.

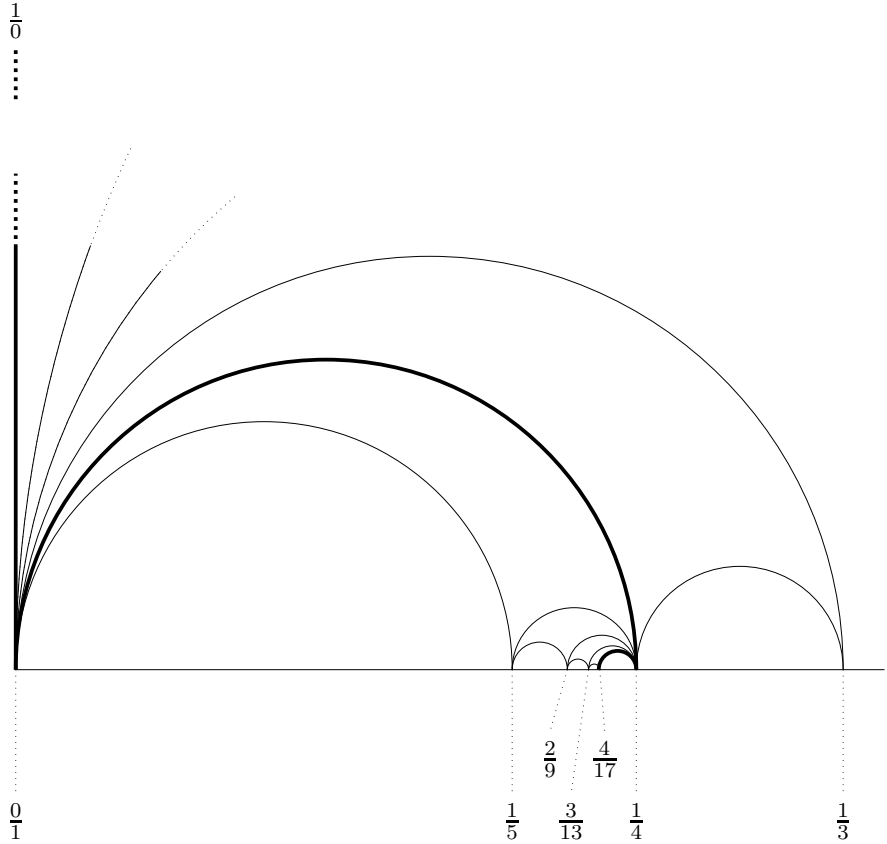
\begin{figure}
\centering
\scalebox{0.2}{}
 \begin{tikzpicture}[scale=33]

\draw [line width=.01 cm] (0,0) -- (.35,0);
\draw [line width=.05 cm] (0,0) -- (0,.17);
\draw [line width=.05 cm, dotted] (0,.17) -- (0,.2);
\draw [line width=.05 cm, dotted] (0,.23) -- (0,.25) node[above] {$\frac10$};
\draw [line width=.01 cm] (0,0) arc (180:160:1/2);
\draw [line width=.01 cm, dotted] (0,0) arc (180:155:1/2);
\draw [line width=.01 cm] (0,0) arc (180:140:1/4);
\draw [line width=.01 cm, dotted] (0,0) arc (180:130:1/4);
\draw [line width=.05 cm] (1/4,0) arc (0:180:1/8);
\draw [line width=.01 cm] (1/3,0) arc (0:180:1/6);
\draw [line width=.01 cm] (1/4,0) arc (0:180:1/40);
\draw [line width=.01 cm] (1/4,0) arc (0:180:1/72);
\draw [line width=.01 cm] (1/4,0) arc (0:180:1/104);
\draw [line width=.05 cm] (1/4,0) arc (0:180:1/132);
\draw [line width=.01 cm] (1/5,0) arc (0:180:1/10);
\draw [line width=.01 cm] (2/9,0) arc (0:180:1/90);
\draw [line width=.01 cm] (3/13,0) arc (0:180:1/234);
\draw [line width=.01 cm] (4/17,0) arc (0:180:1/442);
\draw [line width=.01 cm] (1/3,0) arc (0:180:1/24);

\draw [line width=.01 cm, dotted] (0,0) -- (0,-.05) node[below] {$\frac01$};
\draw [line width=.01 cm, dotted] (1,0) -- (1,-.05) node[below] {$\frac11$};
\draw [line width=.01 cm, dotted] (1/3,0) -- (1/3,-.05) node[below] {$\frac13$};
\draw [line width=.01 cm, dotted] (1/4,0) -- (1/4,-.05) node[below] {$\frac14$};
\draw [line width=.01 cm, dotted] (1/5,0) -- (1/5,-.05) node[below] {$\frac15$};
\draw [line width=.01 cm, dotted] (2/9,0) -- (2/9-.007,-.025) node[below] {$\frac29$};
\draw [line width=.01 cm, dotted] (3/13,0) -- (3/13-.007,-.05) node[below] {$\frac{3}{13}$};
\draw [line width=.01 cm, dotted] (4/17,0) -- (4/17+.002,-.025) node[below] {$\frac4{17}$};

     \end{tikzpicture}
  \caption{A picture of the initial segment of the geodesic in the curve graph $\cC(S_h)$ (again visualized as the Farey graph) defined by the continued fraction $[0;4,4,4,\ldots]$.  The edges of the geodesic are drawn as thicker lines, and the first few vertices, $\gamma_{-1}^h=\frac{1}0, \gamma_0^h = \frac{0}1, \gamma_1^h = \frac{1}4, \gamma_2^h = \frac{4}{17}, \ldots$, are indicated. 
As one can see from the first few segments of the geodesic, it ``pivots'' on opposite sides (c.f.~\cite{geometricccplx}), reflecting the fact that the sign in front of $e^h_i = 4 > 0$ in the Dehn twisting alternates.   
It follows that the segment $[\gamma^h_{i},\gamma^h_{i+1}]$ is separated from the segment $[\gamma^h_0,\gamma^h_1]$ by
 the segment $[\gamma^h_{i-1},\gamma^h_{i}]$ in $\cC(S_h)$.}
\label{fig:farey2}
  \end{figure}

The curve graph $\cC(S_h)$---or equivalently the Farey graph---naturally embeds into the Weil-Petersson completion of $\Teich(S_h)$ in such a way that the vertex corresponding to the curve $\gamma$ is sent to the point in which $\gamma$ has been pinched, and so that edges between adjacent vertices are sent to WP geodesics.  
Furthermore, the pants graph and the curve graph of a once punctured torus coincide, and according to \cite[Theorem 3.2]{br} this embedding is a quasi-isometry.  The usual identification of $\oT{S_h}$ with a subset of the compactified upper half-plane provides the standard embedding of the Farey graph into $\overline{\mathbb H^2}$, with vertex set $\mathbb Q \cup \{\infty\} \subset \mathbb R \cup \{\infty\} = S^1_\infty$.
We further note that all maps and identificiitons are equivariant with respect to the actions of $\Mod(S_h) \cong \mbox{SL}_2(\mathbb Z)$ on the various graphs/spaces.

For each $i \geq 0$, let $X_i^h \in \oT{S_h}$ denote the point at which $\gamma_i^h$ is pinched and $[X_i^h,X_i^{h+1}]$ the geodesic in $\oT{S}$ between points $X^h_i$ and $X^h_{i+1}$. These geodesic segments are the images of the geodesics $[\gamma_i^h,\gamma_{i+1}^h]$ in $\cC(S_h)$ we described above, and since the concatenation of the latter set of segments is a geodesic in $\cC(S_h)$, the image is a quasi-geodesic in $\oT{S_h}$. 
Then since the action of $\Mod(S_h)$ on $\gamma^h_i$ is transitive, it is clearly
 transitive on the geodesics segments $[X^h_i,X^h_{i+1}]$. Moreover, $\Mod(S_h)$ acts isometrically on $\Teich(S_h)$, 
 so all geodesics $[X^h_i,X^h_{i+1}]$ have the same lengths, 
we denote the length by
\begin{equation} \label{Eq:defining D} D = d_{\WP}(X_i^h,X_{i+1}^h) > 0. \end{equation}
Note that $\gamma_0^0 = \gamma_0^1$ is the curve corresponding to the rational number $0$, and for convenience we let $X_{-1}^h$ 
denote the midpoint of the geodesic segment between (the image of) $\gamma_0^h$ and $\gamma_{-1}^h$ 
which has distance $\frac{D}{2}$ to $X_0^h$ (note that $X_{-1}^h =\sqrt{-1}$ in the upper half plane).

Let $\hat r^h_c$ be the unit speed parameterization of the concatenation of segments $[X^h_i,X^h_{i+1}],\; i\in \mathbb{N}$, for $h = 0,1$.  The set of (not necessarily infinite) geodesic rays starting at $X_{-1}^h$ and passing through a point on the geodesic segment $[X^h_i,X^h_{i+1}]$ forms a nested sequence, indexed by $i$. 
To see this, note that by the change of the sign of the power of $D_{\gamma^h_i}$ in (\ref{eq : tw gh i-1, i+1}),  the geodesic $[\gamma^h_{i},\gamma^h_{i+1}]$ is separated from $\gamma^h_{-1}$ by the geodesic $[\gamma^h_{i-1},\gamma^h_{i}]$ in $\cC(S_h)$; see Figure \ref{fig:farey2}. This implies that the geodesic $[X^h_{i},X^h_{i+1}]$ is separated from $X^h_{-1}$ by the geodesic $[X^h_{i-1},X^h_{i}]$ in $\overline{\Teich(S_h)}$, and hence any geodesic starting from $X^h_{-1}$ that passes through $[X^h_{i},X^h_{i+1}]$ must also pass through $[X^h_{i-1},X^h_{i}]$.

Now note that $\hat r^h_c$ is an infinite quasi-geodesic in $\overline{\Teich(S_h)}$, so the distance between the segments $[X^h_i,X^h_{i+1}]$ and $X^h_{-1}$ go to infinity. Then the negative curvature of the WP metric on $\Teich(S_h)$ \cite[Corollary 7.6]{wol}
implies that the maximum of the smaller angles at $X^h_{-1}$ between any two geodesics in the nested sequence of geodesic segments tends to $0$ as $i\to\infty$. This guarantees
 the existence of a unique ray in the intersection of all these sets. We denote the ray by $\hat r^h$ and note that it fellow travels $\hat r^h_c$.

\begin{lem} \label{lem : close to concatenation} 
There exists a sequence $\{K_i\}_{i=1}^\infty$ so that if $e_i^h > K_i$ for all $i \geq 0$, 
and if $\{t_i^h\}$ is the sequence of times for which $d_{\WP}(\hat r^h(t_i^h),X^h_i)$ is minimized, then for $D$ from (\ref{Eq:defining D}) we have
\begin{enumerate}
\item $d_{\WP}(\hat r^h_c(s),\hat r^h(s)) \leq \tfrac{D}4$ for all $s > 0$,\\
\item $d_{\WP}(\hat{r}^h(t^h_i),X^h_i) \leq \tfrac{D}{2^{i+6}}$ (which tends to $0$ as $i\to\infty$), and\\
\item $\left| t_i^h - \left( \frac{D}2 + iD \right) \right| < \tfrac{D}8$ (in particular, $\{t_i^h\}_i$ is increasing).\\
\end{enumerate}

\end{lem}
\begin{proof} 
First we will show that for all $i \geq 0$, we can choose $K_i > 0$ such that if $e_i^h > K_i$, then we have
\[ d_{\WP}(\hat r^h(t_i^h),X_i^h) < \tfrac{D}{2^{i+6}}.\]
To prove this, first let $\delta_i^h$ denote the segment of $\hat r^h$ with one endpoint on $[X^h_{i-2},X^h_{i-1}]$ and and the other on $[X_{i+1}^h,X_{i+2}^h]$ (recall that $\hat r^h$ pass through these geodesics).
Since the piecewise geodesic segment $\hat r^h_c$ contains a segment of length $4D$ containing $[X^h_{i-2},X^h_{i-1}]$ and $[X_{i+1}^h,X_{i+2}^h]$, the length of $\delta_i^h$ is at most $4D$.
Given any $\eta > 0$, we claim that there exists $C(\eta) > 0$ so that if $\hat r^h$ stays outside of the $\eta$--neighborhood of $X_i^h$, then the length of $\delta_i^h$ is at least $C(\eta)e_i^h$.  If we prove this claim, then taking $\eta_i = \tfrac{D}{2^{i+6}}$, we can set $K_i > \frac{4D}{C(\eta_i)}$ and observe that if $e_i^h > K_i$, then $\delta_i^h$ (and hence $\hat r^h$) must enter the $\eta_i$--neighborhood, as required.

To prove the claim, recall that distance from a point $X \in \Teich(S_h)$ to $X_i^h$ is $(2 \pi \ell_{\gamma_i^h})^{\frac12} + O(\ell_{\gamma_i^h}^2)$ (see \cite[Corollary 4.10]{wolb}).  In particular, there exists $L(\eta) > 0$ so that the sublevel set $\ell_{\gamma_i^h}^{-1}((0,L(\eta)])$ is contained in the the ball of radius $\eta$ about $X_i^h$.  By convexity of length functions, \cite[\S 3.3]{wol}, the set $\ell_{\gamma_i^h}^{-1}((0,L(\eta)])$ is convex and hence the closest point projection of $\delta_i^h$ to it is no longer than $\delta_i^h$.  The length of each arc of the boundary of $\ell_{\gamma_i^h}^{-1}((0,L(\eta)])$ intersected with a triangle of the Farey tessellation is some constant $C(\eta)>0$, and since the projection of $\delta_i^h$
to the sublevel set has to cross at least $e_i^h$ of these arcs, its length is at least $C(\eta)e_i^h$, as required; see Figure~\ref{fig:farey3}.

\begin{figure}
\centering
\scalebox{0.2}{}
 \begin{tikzpicture}[scale=1.4]

\draw[fill=gray!50, draw = none] (-1.5,1.5) -- (7.5,1.5) -- (7.5,2.5) -- (-1.5,2.5) -- (-1.5,1.5);
\draw [line width = .001cm] (-1.5,1.5) -- (7.5,1.5);

\draw (-1.5,0) -- (7.5,0);
\draw (0,0) -- (0,2.5);
\draw (1,0) -- (1,2.5);
\draw (2,0) -- (2,2.5);
\draw (4,0) -- (4,2.5);
\draw (5,0) -- (5,2.5);
\draw (6,0) -- (6,2.5);
\draw (0,0) arc (0:180:1/6);
\draw (1,0) arc (0:180:1/2);
\draw (2,0) arc (0:180:1/2);
\draw (2,0) arc (180:90:1/2);

\node at (3,.75) {$\cdots \cdots$};

\draw (4,0)  arc (0:90:1/2);
\draw (5,0) arc (0:180:1/2);
\draw (6,0) arc (0:180:1/2);
\draw (19/3,0) arc (0:180:1/6);

\draw[dotted] (-1/3,0) -- (-1/3,-.5) node[below]{\tiny $X_{i-2}^h$};
\draw[dotted] (0,0) -- (0,-.75) node[below]{\tiny $X_{i-1}^h$};
\draw[dotted] (6,0) -- (6,-.75) node[below]{\tiny $X_{i+1}^h$};
\draw[dotted] (19/3,0) -- (19/3,-.5) node[below]{\tiny $X_{i+2}^h$};

\draw[<-] (0,-.5) -- (2.75,-.5);
\draw[->] (3.25,-.5) -- (6,-.5);
\node at (3,-.5) {$e_i^h$};

\node at (-.75,2) {$\ell_{\gamma_i^h}^{-1}(L(\eta))$};
\node at (3,2.75) {$X_i^h$};

\end{tikzpicture}
  \caption{Applying an element of $\Mod(S_h)$ we can assume $X_i^h$ is the point at infinity in the upper half-plane model of $\oT{S_h}$, as shown.  Then for $L(\eta) > 0$ small, $\ell_{\gamma_i^h}^{-1}((0,L(\eta)])$ is approximately a horoball, since a horoball is the sublevel set of the extremal length function and since hyperbolic lengths and extremal lengths are nearly proportional for small values (see \cite{maskit}).}
\label{fig:farey3}
  \end{figure}
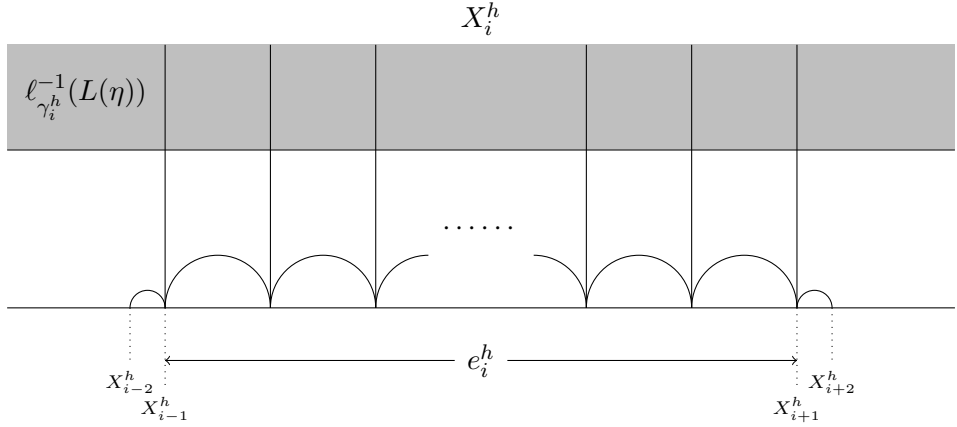


We now assume (as we will for the remainder of the proof) that $e_i^h>K_i$ for all $i \geq 0$, and observe that part (2) of the lemma holds.

By the triangle inequality, it follows that for all $i \geq 0$
\begin{equation} \label{E: nearly equal to D}
\begin{array}{rcl} \displaystyle{\Big| |t^h_i - t^h_{i+1}| - D \Big|} & = & \displaystyle{\left| d_{\WP}(\hat r^h(t_i^h),\hat r^h(t_{i+1}^h)) - d_{\WP}(X_i^h,X_{i+1}^h) \right|}\\
& \leq &\displaystyle{ \tfrac{D}{2^{i+6}} + \tfrac{D}{2^{i+7}} < \tfrac{D}{2^{i+5}}}.
\end{array}
\end{equation}
A similar (simpler) argument proves $|t_0^h-\frac{D}2| < \tfrac{D}{2^6}$.

We claim that for all $i\geq 1$, $t_{i+1}^h$ must lie between $t_i^h$ and $t_{i+2}^h$.  If not, and for example $t^h_{i+1} > \max\{t^h_i,t^h_{i+2}\}$, then $t^h_{i+1} - t^h_i > 0$ and $t^h_{i+1}-t^h_{i+2} > 0$, and applying inequality (\ref{E: nearly equal to D}) to $i$ and $i+1$, we see that
\[ |t_i^h-t_{i+2}^h| = |t_{i+1}^h-t_{i+2}^h - D  + (D-(t_{i+1}^h-t_i^h))| \leq \tfrac{D}{2^{i+5}} + \tfrac{D}{2^{i+6}}.\]
Hence by the triangle inequality (as above)
\begin{eqnarray*} d_{\WP}(X_i^h,X_{i+2}^h) & \leq & d_{\WP}(\hat r^h(t_i^h),\hat r^h(t_{i+2}^h)) + \tfrac{D}{2^{i+6}} + \tfrac{D}{2^{i+8}} \\
& \leq & |t_i^h-t_{i+2}^h| + \tfrac{D}{2^{i+6}} + \tfrac{D}{2^{i+8}}  < \tfrac{D}4.
\end{eqnarray*}
On the other hand, \cite[Lemma 3.2]{brockmargalit} implies that since $\gamma_i^h$ and $\gamma_{i+2}^h$ are not adjacent in $\cC(S_h)$, we must have $d_{\WP}(X_i^h,X_{i+2}^h) > D$, which is a contradiction.  A similar argument produces a contradiction if $t_{i+1}^h <\min\{ t_i^h,t_{i+2}^h\}$, hence as we claimed $t^h_{i+1}$ is between $t^h_i$ and $t^h_{i+2}$, and thus
 $\{t_i^h\}_i$ is an increasing sequence.

From (\ref{E: nearly equal to D}) (and the inequality $|t_0^h-\frac{D}2| < \tfrac{D}{2^6}$), we have
\begin{eqnarray*}
\Big| t_i^h - \left( \tfrac{D}2 + iD \right) \Big| & = & \Big| t_0^h + \sum_{j=1}^i t_j^h - t_{j-1}^h - \tfrac{D}2 - iD \Big| \\
& \leq & \Big| t_0^h - \tfrac{D}2 \Big| + \sum_{j=1}^i \Big| t_j^h - t_{j-1}^h - D \Big| \leq \tfrac{D}{2^6} + \sum_{j=1}^i \tfrac{D}{2^{j + 4}} < \tfrac{D}8.
\end{eqnarray*}
This proves part (3).

Finally, we note that part (1) follows from (\ref{E: nearly equal to D}), parts (2) and (3), and convexity of distance between two geodesics in a $\CAT(0)$ space.
 To see this, first note that for all $i \geq 0$
\begin{eqnarray*} 
d_{\WP}(\hat r^h(\tfrac{D}2 + iD),\hat r^h_c(\tfrac{D}2 + iD)) \!\!\!\! & \leq & \!\!\!\! d_{\WP}(\hat r^h(\tfrac{D}2 + iD),\hat r^h(t_i^h)) \! + \! d_{\WP}(\hat r^h(t_i^h),X_i^h)\\
& \leq &  \!\!\!\! |t_i^h - (\tfrac{D}2 + iD)| + \tfrac{D}{2^{i+6}} < \frac{D}8 + \frac{D}{2^{i+6}} < \tfrac{D}4.
\end{eqnarray*}
Thus, for all $i \geq 0$, convexity of the distance between geodesics implies
\[ d_{\WP}(\hat r^h(s),\hat r^h_c(s)) < \tfrac{D}4, \quad \text{for all}\; s \in \left[\tfrac{D}2 + iD,\tfrac{D}2 + (i+1)D \right] \]
(and for $s \in [0,\tfrac{D}2]$).  This proves (1) for all $s \geq 0$, completing the proof.
\end{proof}

\subsection{Sequences of times}

Throughout the following, we will always assume that for each $h = 0,1$, the sequence $\{e^h_i\}_i$ is chosen so that $e^h_i > K_i$ from Lemma~\ref{lem : close to concatenation}, and we write $\hat r^h,\hat r^h_c$ to denote the associated geodesics/quasi-geodesics.
We keep the same parameterization for $\hat r^0$ and $\hat r^0_c$ as above, but adjust the parameterization of $\hat r^1$ and $\hat r^1_c$ by precomposing with the maps $t \mapsto t-\tfrac{D}2$.  
This does not make sense for $t \in [0,\tfrac{D}2)$, so we define $\hat r^1$ and $\hat r^1_c$ to be constant on this interval. 

With this new parameterization, the sequences $\{t_i^1\}$ must be shifted by $\tfrac{D}2$, so that parts (1) and (2) of Lemma~\ref{lem : close to concatenation} remain valid.  
The conclusion in part (3) of the lemma then becomes
\begin{equation} \label{lem':new part (3)}
|t_i^0 - \left( \tfrac{D}2 + iD \right) | < \tfrac{D}8 \quad \mbox{ and } \quad  |t_i^1 - (i+1)D| < \tfrac{D}8.
\end{equation}

Identifying $\cS(\alpha) = \Teich(S_0) \times \Teich(S_1)$, we set
\[ \hat r = (\hat r^0,\hat r^1) \colon [0,\infty) \to \cS(\alpha) \subset \oT{S}.\]

\begin{notation}(Relabeling sequences)\label{not : seq k} 
To simplify some statements and avoid duplication in some of the arguments that follow, we make the following notational convention.  For $h = 0,1$ and $i \geq 0$, set
\begin{eqnarray*}
e_{2i + h} &=& e_i^h\\
\gamma_{2i+h} &=& \gamma_i^h\\
t_{2i+h} &=& t_i^h\\
X_{2i+h} &=& X_i^h.
\end{eqnarray*}
We will use the index $k$ for these sequences, and write $\{e_k\}$, $\{\gamma_k\}$, $\{t_k\}$, and $\{X_k\}$.  
We also let $\bar k \in \{0,1\}$ denote the residue of $k$ modulo $2$, and $i = i(k)$ for the floor of $k/2$.  Thus, when we need it, can write $e_k = e_i^{\bar k}$, etc.  
As an abuse of notation, we say things like ``$\hat r(t_k)$ is close to $X_k$'', though what we really mean is that $\hat r^{\bar k}(t_k)$ is close to $X_k$.  
We also view $\gamma_k$ as a curve on both $S$ and $S \setminus \alpha$, rather than just a curve on $S_{\bar k} \subset S \setminus \alpha \subset S$.
Finally, the following sequence of times will also be useful for us
\[ t_k' = \frac{t_k + t_{k+1}}2. \]
\end{notation}

\begin{prop} \label{prop : monotonicity and separation}
For all $k \geq 0$, $t_{k+1} - t_k > \tfrac{D}4$.  In particular,  $\{t_k\}_k$ is increasing.
Consequently, $t_k'-t_k > \tfrac{D}8$ and $t_{k+1}-t_k' > \tfrac{D}8$.
\end{prop}
\begin{proof}  For $k=2i$ (even), (\ref{lem':new part (3)}) implies
\[ t_{k+1}- t_k = t_i^1 - t_i^0 > \left( (i+1)D - \tfrac{D}8\right) - \left(\tfrac{D}2 + iD + \tfrac{D}8 \right) = \tfrac{D}2 - \tfrac{D}4 = \tfrac{D}4.\]
A similar computation verifies the claim for $k$ odd.

The last sentence follows from the first, and the fact that $t_k'$ is the average of $t_k$ and $t_{k+1}$.
\end{proof}
Figure~\ref{F:timing} provides a useful illustration of the relationship between $\{t_k\}$, $\{t_k'\}$, $\{\gamma_k\}$, and $\{X_k\}$.

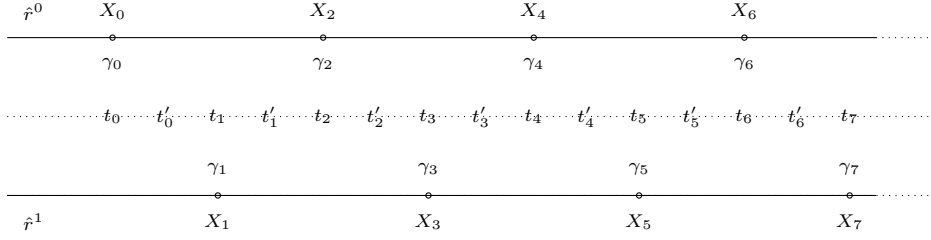
\begin{figure}[h]
\begin{center}
\begin{tikzpicture}[scale = .7]
\draw (0,1.5) --  (16.5,1.5);
\draw[dotted] (16.5,1.5) -- (17.5,1.5);
\draw (0,-1.5) --  (16.5,-1.5);
\draw[dotted] (0,-1.5) -- (17.5,-1.5);
\draw[dotted] (0,0) -- (17.5,0);
\draw (2,1.5) circle (.05cm);
\draw (4,-1.5) circle (.05cm);
\draw (6,1.5) circle (.05cm);
\draw (8,-1.5) circle (.05cm);
\draw (10,1.5) circle (.05cm);
\draw (12,-1.5) circle (.05cm);
\draw (14,1.5) circle (.05cm);
\draw (16,-1.5) circle (.05cm);
\node at (2,2) {\tiny $X_0$};
\node at (4,-2) {\tiny $X_1$};
\node at (6,2) {\tiny $X_2$};
\node at (8,-2) {\tiny $X_3$};
\node at (10,2) {\tiny $X_4$};
\node at (12,-2) {\tiny $X_5$};
\node at (14,2) {\tiny $X_6$};
\node at (16,-2) {\tiny $X_7$};
\node at (.5,2) {\tiny $\hat r^0$};
\node at (.5,-2) {\tiny $\hat r^1$};
\node at (2,1) {\tiny $\gamma_0$};
\node at (4,-1) {\tiny $\gamma_1$};
\node at (6,1) {\tiny $\gamma_2$};
\node at (8,-1) {\tiny $\gamma_3$};
\node at (10,1) {\tiny $\gamma_4$};
\node at (12,-1) {\tiny $\gamma_5$};
\node at (14,1) {\tiny $\gamma_6$};
\node at (16,-1) {\tiny $\gamma_7$};
\node at (2,0) {\tiny $t_0$};
\node at (3,0) {\tiny $t_0'$};
\node at (4,0) {\tiny $t_1$};
\node at (5,0) {\tiny $t_1'$};
\node at (6,0) {\tiny $t_2$};
\node at (7,0) {\tiny $t_2'$};
\node at (8,0) {\tiny $t_3$};
\node at (9,0) {\tiny $t_3'$};
\node at (10,0) {\tiny $t_4$};
\node at (11,0) {\tiny $t_4'$};
\node at (12,0) {\tiny $t_5$};
\node at (13,0) {\tiny $t_5'$};
\node at (14,0) {\tiny $t_6$};
\node at (15,0) {\tiny $t_6'$};
\node at (16,0) {\tiny $t_7$};
\end{tikzpicture}
\caption{The times $t_k$ and $t_k'$ are ``spaced out'' by at least $\tfrac{D}8$.  The former are the times when $\hat r = (\hat r^0,\hat r^1)$ is closest to the points $\{X_k\}$: the curve $\gamma_k$ is very short at time $t_k$.}
\label{F:timing}
\end{center}
\end{figure}

\begin{lem}  \label{lem : bounded length marking times} There exists $C > 0$, so that for all $k \geq 2$, we have
\begin{center} $\ell_{\gamma_k}(\hat r(t_{k-2}'))\leq C$ and $\ell_{\gamma_k}(\hat r(t_{k+1}')) \leq C$.\end{center}
Consequently, $\ell_{\gamma_k}(\hat r(t)) \leq C$ for all $t \in [t_{k-2}',t_{k+1}']$
\end{lem}
\begin{proof} We prove the bound on $\ell_{\gamma_k}(\hat r(t_{k-2}'))$.  The proof of the other bound is similar.

According to part (2) of Lemma~\ref{lem : close to concatenation}, for $k \geq 2$
\[ d_{\WP}(\hat r^{\bar k}(t_{k-2}),X_{k-2})\leq \tfrac{D}{2^6}\;\text{and}\; d_{\WP}(\hat r^{\bar k}(t_k),X_k) \leq \tfrac{D}{2^6}. \]
By convexity of distance between geodesics, it follows that there is a point $Y_k \in [X_{k-2},X_k]$ such that
\[ d_{\WP}(\hat r^{\bar k}(t_{k-2}'),Y_k) \leq \tfrac{D}{2^6}.\]

On the other hand, by Proposition~\ref{prop : monotonicity and separation} we have
\begin{eqnarray*}
 t_{k-2}+\tfrac{D}8 < t_{k-2}' &=&(t'_{k-2}-t_{k-1})+(t_{k-1}-t'_{k-1})+(t'_{k-1}-t_k)+t_k \\
 &<& t_k- \tfrac{3D}8 .
 \end{eqnarray*}

Therefore, by the triangle inequality, we see that
\begin{eqnarray*}
d_{\WP}(Y_k,X_{k-2}) & \geq & \!\! (t_{k-2}' - t_{k-2}) -  \! d_{\WP}(\hat r^{\bar k}(t_{k-2}'),Y_k)  - \! d_{\WP}(X_{k-2},\hat r^{\bar k}(t_{k-2}))\\
& > & \!\!  \tfrac{D}8  - \tfrac{D}{2^6} - \tfrac{D}{2^6} > \tfrac{D}{16}
\end{eqnarray*}
Similar computations show that
\[ d_{\WP}(Y_k,X_k) \geq \tfrac{D}4.\]
So, $Y_k$ is further than $\tfrac{D}{16}$ from the endpoints of $[X_{k-2},X_k]$, and so less than $\tfrac{7D}{16}$ from the midpoint.  In particular, the closed ball of radius $\tfrac{D}{32}$ in $\oT{S_{\bar k}}$ about $Y_k$ is contained in the closed ball $B_k \subset \oT{S_{\bar k}}$ of radius $\tfrac{15D}{32}$ centered at the midpoint $M_k$ of $[X_{k-2},X_k]$.

We claim that $B_k \subset \Teich(S_{\bar k})$ (that is, $B_k$ contains no completion points), and hence $B_k$ is compact.
To prove the claim, it suffices to show that the closest point to $M_k$ in $\oT{S_{\bar k}} \setminus \Teich(S_{\bar k})$ is one of the endpoints $X_{k-2}$ or $X_k$.
For this, let
\[ X \in \oT{S_{\bar k}} \setminus (\Teich(S_{\bar k}) \cup \{X_{k-2},X_k\})\]
be any completion point.  According to \cite[Lemma~3.2]{brockmargalit}, we have that  $D = d_{\WP}(X_{k-2},X_k) \leq d_{\WP}(X_k,X)$.  
Since triangles in $\oT{S_{\bar k}}$ are nondegenerate (meaning that edges meet only in a vertex), $M_k$ is not contained in the geodesic segment $[X_k,X]$.  Thus, the (strict) triangle inequality implies
\[ 2 d_{\WP}(M_k,X_k) =  D  \leq d_{\WP}(X_k,X) < d_{\WP}(X_k,M_k)  + d_{\WP}(M_k,X).\]
Therefore $d_{\WP}(M_k,X_k) < d_{\WP}(M_k,X)$.

We now see that the closed ball of radius $\tfrac{D}{32}$ about $Y_k$ is contained in the $\Mod(S_{\bar k})$--orbit of a single compact set in $\Teich(S_{\bar k})$, namely the closed ball of radius $\tfrac{15D}{32}$ about the midpoint of a single Farey edge.
Therefore, the length of $\gamma_k$ (the curve pinched at $X_k$) is uniformly bounded in the $\tfrac{D}{32}$--neighborhood of $Y_k$, independent of $k$ (and independent of the sequence $\{e_k\}$).  Since $\hat r^{\bar k}(t_{k-2}')$ lies in this neighborhood, $\ell_{\gamma_k}(\hat r^{\bar k}(t_{k-2}'))$ is uniformly bounded, as required.

The proof of the bound on $\ell_{\gamma_k}(\hat r^{\bar k}(t_{k+1}'))$ is entirely analogous, using the geodesic segment $[X_k,X_{k+2}]$ in place of $[X_{k-2},X_k]$.   
The very last statement follows from convexity of length-functions along WP geodesics \cite[\S 3.3]{wol}.
\end{proof}

\begin{cor} \label{cor:special hat times} For all $k \geq 2$ and $j = k-1,k,k+1,k+2$, we have
\[ \ell_{\gamma_j}(\hat r(t_k')) \leq C. \]
\end{cor}
\begin{proof} According to Lemma~\ref{lem : bounded length marking times}, the curve $\gamma_j$ has length at most $C$ on the interval $[t_{j-2}',t_{j+1}']$.  The corollary thus follows from the fact that
\[ \{t_k'\} = \bigcap_{j=k-1}^{k+2} [t_{j-2}',t_{j+1}'].\]
\end{proof}

\subsection{Intersection number estimates}

We will require the following estimate for the intersection number of a curve $\delta$ and the curves $\gamma^h_i$ in terms of the numbers $e^h_i, \; i\geq 0$.

\begin{lem}\label{lem : intersection}
Given $\delta \in \cC_0(S)$ with $\I(\delta,\alpha) \neq 0$, there exists $\kappa = \kappa(\delta) \geq 1$ so that for $h=0,1$ and all $i$ sufficiently large we have 
\[\frac{1}{\kappa}e^h_{i-1}\leq \I(\delta,\gamma^h_i)\leq \kappa I_h(i),\]
for $I_h(i)=\sum_{J} \prod_{j\in J} e^h_j$ where $J$ runs over all subsets of $\{0,\ldots,i-1\}$ exactly once.
\end{lem}
\begin{proof}

 Suppose that $\delta\cap S_h$, $h=0,1$, consists of $n_h$ geometric arcs with end points on $\alpha=\partial S_h$ 
 (geometric arcs are proper arcs on the surface and homotopic geometric arcs are not identified).
Let $[0;e^h_1,e^h_2,\ldots, e^h_{i-1},\ldots]$ be a continued fraction expansion as in $\S$\ref{subsec : contfrac} 
and recall that the curve $\gamma^h_i$ has slope reciprocal to $\frac{p^h_i}{q^h_i}$ the $i^{th}$ convergent of the continued fraction expansion.
 Let $\tau_h$ be a geometric arc in $\delta\cap S_h$ with the largest intersection number with $\gamma^h_i$ and let $\frac{a_h}{b_h}$ be the reciprocal of the slope of $\tau_h$.
  Then, since $\gamma^h_i\subset S_h$, we have that $\I(\tau_h,\gamma^h_i)= |a_hq^h_i-b_hp^h_i|$: to see this, observe that orienting $\gamma^h_i$ and $\tau_h$, these represent the (relative) homology classes $(a_h,b_h)$ and $(p^h_i,q^h_i)$, respectively, in $H_1(S_h,\partial S_h; \mathbb Z) \cong Z^2$, and $\I(\tau_h,\gamma^h_i)$ is the absolute value of the algebraic intersection number, which is the geometric intersection number on a punctured torus.

 The standard recursive formula for convergents of continued fraction expansions gives us
 $q^h_{i}=e^h_{i-1}q^h_{i-1}+q^h_{i-2}$ (see e.g. \cite[Theorem 1]{Khinchin-contfrac}; recall our index convention in $\S$\ref{subsec : contfrac}), we also have  
  that $q^h_0=1$ and $q^h_1=e^h_0$.
 Then we can easily verify by induction on $i$ that 
 \begin{equation}\label{eq : qhi}
 q^h_i=\sum_{J\subseteq \{0,\ldots,i-1\}}\prod_{j\in J}e^h_j
 \end{equation}
  where each subset $J$ appears at most once in the sum. 
  Now since $\lim_{i\to\infty}\frac{p^h_i}{q^h_i}=x_h$, where the irrational number  $x_h$ is the reciprocal of the slope of $\lambda_h$, we have 
\[
\lim_{i\to\infty}\frac{|a_hq^h_i-b_hp^h_i|}{q^h_i}=\lim_{i\to\infty}|a_h-\left(\frac{p^h_i}{q^h_i}\right)b_h|= |a_h-x_hb_h|.
\]
Thus, for $i$ sufficiently large 
\[
|a_hq^h_i-b_hp^h_i|\leq 2|a_h-x_hb_h| q^h_i\leq 2|a_h-x_hb_h| I_h(i).
\]
Then since $\tau_h$ is a geometric arc with the largest intersection number with $\gamma^h_i$ and since there are $n_h$ geometric arcs in $\delta\cap S_h$ we have 
\begin{equation}\label{eq : ub}
\I(\delta,\gamma^h_i)\leq 2n_h|a_h-x_hb_h| I_h(i).
\end{equation}
 Furthermore, by (\ref{eq : qhi}), $q^h_i\geq \prod_{j=0}^{i-1}e^h_j\geq e^h_{i-1}$. From this inequality and the above limit we deduce that the inequality
\begin{equation}\label{eq : lb}
|a_hq^h_i-b_hp^h_i|\geq \frac{1}{2}|a_h-x_hb_h| q^h_i\geq \frac{1}{2}|a_h-x_hb_h| e^h_{i-1}
\end{equation}
holds for all $i$ sufficiently large. 

Now from inequalities (\ref{eq : ub}) and (\ref{eq : lb}) we see that 
the inequalities of the lemma hold for $\kappa= \max\{2n_h|a_h-x_hb_h|,\frac{2}{|a_h-x_hb_h|}:h=0,1\}$.
\end{proof}

For any $k\in\mathbb{N}$, appealing to the Notation~\ref{not : seq k}, let $I(k)=I_{\bar{k}}(i)$ where $k=2i+\bar{k}$. The conclusion of Lemma~\ref{lem : intersection} then becomes
\begin{equation} \label{eqn:intersection index change}
\frac1\kappa e_{k-2} \leq i(\delta,\gamma_k) \leq \kappa I(k).
\end{equation}
For the remainder of the paper, we assume that the sequence $\{e_k\}_k$ satisfies the additional growth condition 
\begin{equation}\label{eq : gr1}
\lim_{k \to \infty} \frac{I(k)}{e_k} = 0, \;\text{and}\;  \lim_{k\to\infty}\frac{I(k+1)}{e_k}=0. 
\end{equation}
This is possible since $I(k)$ depends only on $\{e_j\}_{j=0}^{k-2}$.

With this convention, we have the following corollary of Lemma~\ref{lem : intersection}.
\begin{cor} \label{cor : rel small intersect}
For any curve $\delta \in \cC_0(S)$ with $\I(\delta,\alpha) \neq 0$ we have
\[ \lim_{k \to \infty} \frac{\I(\delta,\gamma_k)}{e_k} = 0, \;\text{and}\; \lim_{k \to \infty} \frac{\I(\delta,\gamma_{k+1})}{e_k}=0.\]
\end{cor}

\subsection{Geodesics in $\Teich(S)$ and bounded length curves.}

We begin by recalling \cite[Corollary 3.5]{wpbehavior} and the inequality inside its proof, which we will use in some of the estimates in this section. 

\begin{lem}\label{lem : dl}
Given $c>0$ let $l,a\in[0,c]$ with $l>a$. Suppose that for a curve $\beta\in\cC_0(S)$ and points $X,X'\in\Teich(S)$ we have
 $\ell_\beta(X)\leq l-a$ and $\ell_\beta(X')\geq l$, then 
\[d_{\WP}(X,X')\geq \frac{a}{\sqrt{\frac{2}{\pi}l+O(l^4)}},\]
 where the constant of the $O$--notation depends only on $c$.
\end{lem}

\begin{lem} \label{lem : ep1}
There is an $\ep_1 > 0$ and a $C' > 0$ so that for all points $Y$ in the $\ep_1$--neighborhood of $\hat r^{\bar k}(t_k')$ and all $j = k-1,k,k+1,k+2$, we have
\[ \ell_{\gamma_j}(Y) < C'. \]
\end{lem}
\begin{proof} Let $C$ be the constant from Corollary~\ref{cor:special hat times} so that $\ell_{\gamma_j}(\hat{r}^{\bar k}(t'_k))\leq C$ for $j=k-1,\ldots, k+2$. 
Let $a>0, C'=a+C$ and $c=C'+a+1$.   Then
\[ 0 < a < C' < c \mbox{ and } \ell_{\gamma_j}(\hat r^{\bar k}(t_k')) \leq C = C'-a,\]
for each $j  =k-1,k,k+1,k+2$.
Define
\[ \ep_1=\frac{a}{\sqrt{\frac{2}{\pi}C'+O(C'^4)}} \]
to be as in Lemma~\ref{lem : dl} where the constant of the $O$--notation only depends on $c$. 
Now, if $\ell_{\gamma_j}(Y)\geq C'$ for a point $Y$ in the $\ep_1$--neighborhood of $\hat{r}^{\bar k}(t_k')$ 
and a curve $\gamma_j$ with $j = k-1,k,k+1, k+2$,
 then applying Lemma \ref{lem : dl} we have $d_{\WP}(Y,\hat{r}^{\bar k}(t_k'))\geq \ep_1$ 
 which contradicts the fact that $Y$ is in the $\ep_1$--neighborhood of $\hat{r}^{\bar k}(t_k')$.  Therefore, $\ell_{\gamma_j}(Y) < C'$, for $j = k-1,k,k+1,k+2$, proving the lemma.
\end{proof}
 
Decreasing $\ep_1$ if necessary, we may further assume that for any point $X \in \Teich(S)$ in the $2\ep_1$--neighborhood of $\cS(\alpha)$ and any curve $\gamma$ essentially intersecting $\alpha$, we have $\ell_\gamma(X)$ is uniformly bounded below.  This follows from Lemma~\ref{L:collar} and the fact that the distance to $\cS(\alpha)$ is $(2 \pi \ell_{\alpha})^{\frac12} + O(\ell_{\alpha}^2)$ (see \cite[Corollary 4.10]{wolb}).  In particular, any point in $\oT{S}$ in the $2\ep_1$--neighborhood may only lie on a stratum corresponding to a (possibly empty) multicurve having zero intersection number with $\alpha$. 
    \medskip
 
 Now let $Z\in \Teich(S)$ be a point in the $\ep_1$--neighborhood of $\hat{r}(0)$. 
 Let then $[Z,\hat{r}(t_k')]$ be the geodesic segment connecting $Z$ to $\hat{r}(t_k')$. 
 By Corollary~\ref{cor:special hat times} the curves $\gamma_k,\gamma_{k+1}$ have bounded lengths at $\hat{r}(t_k')$
 and the sequence of curves $\{\gamma^{\bar k}_i\}_i$, $\bar k=0,1$, 
 is a quasi-geodesic in $\cC(S_{\bar k})$ that converges to a point in the Gromov boundary of $\cC(S_{\bar k})$. 
 Moreover, $S\backslash\alpha$ is the union of $S_0$ and $S_1$. 
 Then as in \cite[Lemma 8.1]{wpbehavior} we can show that after possibly passing to a subsequence $[Z,\hat{r}(t_k')]$
  converges uniformly on compact subsets to an infinite ray 
 \[ r :[0,\infty)\to\Teich(S).\] 
Also, note that the construction of $r$ and the $\CAT(0)$ property of the WP metric imply the rays $r$ and $\hat{r}$, $\ep_1$--fellow travel.

The following are straightforward consequences of the results of this section.
\begin{cor} \label{cor : special times r} For all $k \geq 2$ and $j = k-1,k,k+1,k+2$, we have
\[ \ell_{\gamma_j}(r(t_k')) \leq C',\]
where $C' > 0$ is the constant from Lemma~\ref{lem : ep1}.
\end{cor}
\begin{proof} As noted above, the two geodesics rays $r$ and $\hat{r}$, $\ep_1$--fellow travel, and hence $d_{\WP}(\hat r(t_k'),r(t_k')) < \ep_1$.  
Thus, by Lemma \ref{lem : ep1}, for $j = k-1,k,k+1,k+2$, we have
\[ \ell_{\gamma_j}(r(t_k')) \leq C',\]
as desired.
\end{proof}

This, in turn, implies the following:
\begin{cor} \label{cor : general times r} For all $k \geq 2$, $t \in [t_k',t_{k+1}']$, and $j = k,k+1,k+2$,
\[ \ell_{\gamma_j}(r(t)) \leq C' \]
where $C' > 0$ is the constant from Lemma~\ref{lem : ep1}.
\end{cor}
\begin{proof}  By Corollary~\ref{cor : special times r}, we have $\ell_{\gamma_k}(r(t_{k-2}')) \leq C'$ and $\ell_{\gamma_k}(r(t_{k+1}'))\leq C'$.  
By convexity of length-functions \cite[\S 3.3]{wol}, for all $t \in [t_{k-2}',t_{k+1}']$ we have
\[ \ell_{\gamma_k}(r(t)) \leq C'.\]
Since $[t_k',t_{k+1}'] \subset [t_{k-2}',t_{k+1}'] \cap [t_{k-1}',t_{k+2}'] \cap [t_k',t_{k+3}']$, the result follows.
\end{proof}

 \begin{prop}\label{prop : properties of r}
 The length of $\alpha$ is bounded by $\ell_\alpha(Z)$ along $r$.  Furthermore, the ending lamination of $r$ is the lamination $\lambda_0 \cup \lambda_1$ or $\lambda_0\cup \alpha\cup\lambda_1$.
  \end{prop}
 \begin{proof}
 First note that by convexity of $\ell_\alpha$, \cite[\S3.3]{wol}, and the fact that $\ell_\alpha(\hat{r}(t_k')) = 0$, it follows that $\ell_\alpha$ is bounded by $\ell_\alpha(Z)$ on $[Z,\hat r(t_k')]$.  
 Since $r$ is a limit of a subsequence of the geodesics $[Z,\hat r(t_k')]$, the first claim of the proposition holds.

By Corollary~\ref{cor : special times r} the curves $\gamma_k,\gamma_{k+1}$ have bounded length at $r(t'_k)$ for all $k$, 
hence by the definition of ending lamination, $\lambda_0$ and $\lambda_1$ are contained in the ending lamination of $r$. 
Note that the only measurable lamination properly containing $\lambda_0 \cup \lambda_1$ is $\lambda_0\cup \alpha\cup\lambda_1$, and so $\nu^+(r)$ must be one of these two laminations (and it is the latter one if and only if $\ell_\alpha(r(t)) \to 0$ as $t \to \infty$, i.e.~if $\alpha$ is a pinching curve).
\end{proof}

We now turn to estimates for twists about bounded length curves at $r(t_k')$.

\begin{lem} \label{lem : est tw gamma} For any $\delta \in \cC_0(S)$ with $i(\delta,\alpha) \neq 0$, there exists $c = c(\delta) > 0$ such that for all $t \in [t_k',t_{k+1}']$, we have
\[ \tw_{\gamma_k}(\delta,r(t)) \asya_c e_k, \]
and
\[ \tw_{\gamma_{k+1}}(\delta,r(t_k')) \asya_c 1\]
for all but finitely many $k$ (namely, whenever $i(\gamma_k,\delta) \neq 0$ and $i(\gamma_{k+1},\delta) \neq 0$, respectively).
\end{lem}
\begin{proof}
By Corollary~\ref{cor : general times r} we may choose a bounded length marking $\mu$ at $r(t)$ so that $\gamma_k$ is in the base and $\gamma_{k+2}$ projects to the transversal to $\gamma_k$.  
Recall that $\bar k \in \{0,1\}$ is the residue of $k$ modulo $2$.  Avoiding finitely many $k$, $i(\delta,\gamma_k) \neq 0$, and we may apply the triangle inequality.  Doing so we have
\begin{equation} \label{eqn : bounded depending on delta}
\Big|d_{\gamma_k}(\delta,\gamma_{k+2}) - d_{\gamma_k}(\gamma_{\bar k},\gamma_{k+2})\Big| \leq d_{\gamma_k}(\delta,\gamma_{\bar k}).
\end{equation}
Since $\mu|_{S_{\bar k}}$ is a uniformly bounded length marking at $r(t)$, we have uniform errors (independent of $k$) in the following coarse equations. 
First, by (\ref{eq:twist-marking}) we have
\begin{equation}\label{eq:twgk}
 \tw_{\gamma_k}(\delta,r(t)) \asya d_{\gamma_k}(\delta,\mu) = d_{\gamma_k}(\delta,\gamma_{k+2}).
\end{equation}
Since $\gamma_{k+2}=D^{\pm e_k}_{\gamma_k}(\gamma_{k-2})$, it follows from \cite[Equation~(2.6)]{mm2}) that 
\[d_{\gamma_k}(\gamma_{k-2},\gamma_{k+2})\asya e_k.\]
  Furthermore, because $\{\gamma_{\bar k + 2i}\}_i$ are the vertices of a geodesic in $\cC(S_{\bar k})$, by \cite[Theorem~3.1]{mm2}, we have
\begin{equation} \label{eq:dgk}
d_{\gamma_k}(\gamma_{\bar k},\gamma_{k+2}) \asya d_{\gamma_k}(\gamma_{k-2},\gamma_{k+2}) = e_k.
\end{equation}
Moreover, since we are allowing our error $c = c(\delta)$ to depend on $\delta$, we can combine the coarse equations (\ref{eq:twgk})  and (\ref{eq:dgk})
with inequality (\ref{eqn : bounded depending on delta}) and deduce
\[ \tw_{\gamma_k}(\delta,r(t)) \asya e_k.\]
This proves the first coarse equation of the lemma.

To prove the second coarse equation, we note that by Corollary~\ref{cor : special times r}, we may choose our bounded length marking $\mu$ at $r(t_k')$ so that $\gamma_{k+1}$ is a base curve and $\gamma_{k-1}$ projects to a transversal for $\gamma_{k+1}$.
Thus, similar to Equation~(\ref{eq:twgk}), we see that (\ref{eq:twist-marking}) implies $\tw_{\gamma_{k+1}}(\delta,r(t_k')) \asya d_{\gamma_{k+1}}(\delta,\gamma_{k-1})$.  Furthermore, similar to (\ref{eqn : bounded depending on delta}), for $k$ sufficiently large we have
\[ \Big|d_{\gamma_{k+1}}(\delta,\gamma_{k-1}) - d_{\gamma_{k+1}}(\gamma_{\overline{k+1}},\gamma_{k-1})\Big| \leq d_{\gamma_{k+1}}(\delta,\gamma_{\overline{k+1}}),\]
Since $\gamma_{\overline{k+1}}$ and $\gamma_{k-1}$ preceed $\gamma_{k+1}$ in the $\cC(S_{\overline{k+1}})$--geodesic, appealing to \cite[Theorem~3.1]{mm2} again we have
\[ d_{\gamma_{k+1}}(\gamma_{\overline{k+1}},\gamma_{k-1}) \asya 1. \]
Combining these facts just as in the previous paragraph and increasing $c = c(\delta)$ if necessary, we have
\[ \tw_{\gamma_{k+1}}(\delta,r(t_k')) \asya_c 1,\]
which completes the proof of the lemma.
\end{proof}

\subsection{Estimates for the separating curve.}

We will eventually impose additional growth conditions on our sequence $\{e_k\}$ to control the length and twisting about the separating curve $\alpha$.  The next two lemmas are used to determine those conditions.

\begin{lem}\label{lem : lb for alpha}
 There exists a function $f_1:[0,\infty)\to \mathbb R^+$, so that for any geodesic ray $r$ constructed as above (from sequences $\{e_i^h\}_i$, 
 $h =0,1$, beginning at $Z$) we have $\ell_\alpha(r(T))\geq f_1(T)$ for all $T\in[0,\infty)$.  Moreover, there exists such a function $f_1$ which is continuous.
\end{lem}
\begin{proof}
The proof is by contradiction.  If there is no such function $f_1$ (not necessarily continuous), then there would be a sequence of geodesics $\{r_n\}$ starting at $Z$, coming from sequences $\{e_i^h(n)\}_i$, as above, and some $T >0$ so that $\ell_\alpha(r_n(T)) \to 0$ as $n \to \infty$.  Now the idea of the proof is as follows.  Appealing to convexity of $\ell_\alpha$ on $r_n$, we can deduce that $\ell_\alpha(r_n(t)) \to 0$ as $n \to \infty$ for all $t \geq T$.  In particular, choosing any $T' > T$, we can apply Theorem~\ref{thm : geodesic limit} to  $r_n|_{[0,T']}$.  We will see that the curve $\alpha$ is (eventually) present in all the multicurves from the theorem, producing a contradiction to the non-refraction behavior ensured by the Theorem~\ref{thm : geodesic limit}.  We now proceed to the details.



Recall that we have chosen $Z \in \Teich(S)$ and $\ep_1>0$ from Lemma~\ref{lem : ep1} (and the paragraph following its proof) so that the distance to any stratum $\cS(\sigma)$ is at least $\ep_1$ whenever $\sigma$ has nonzero intersection number with $\alpha$.
By Proposition~\ref{prop : properties of r}, $\ell_\alpha(r_n(t)) \leq \ell_\alpha(Z)$ for all $n$ and $t \geq 0$, so since $\lim_{n\to\infty}\ell_{\alpha}(r_n(T))=0$,  convexity of $\ell_\alpha$ implies $\ell_{\alpha}(r_n(t))\leq \ell_{\alpha}(r_n(T))$ for all $n$ sufficiently large and all $t \geq T$.
In particular, $\lim_{n\to\infty}\ell_{\alpha}(r_n(t))=0$ for all $t \geq T$, while $\ell_\alpha(r_n(0)) = \ell_\alpha(Z) > 0$.
 
Now fix any $T'>T$ and apply Theorem~\ref{thm : geodesic limit} to the sequence of geodesic segments $r_n|_{[0,T']}$.  Let the partition $0=t_0<t_1<\ldots<t_{k+1}=T'$, the piecewise geodesic path $\hat\zeta \colon [0,T'] \to \oT{S}$, the multicurves $\{\sigma_l\}_{l=1}^{k+1}$, the multitwists $\{\cT_{l,n}\}_{l=1}^k$, and the mapping classes $\{\varphi_{l,n}\}_{l=1}^k$ obtained by composing the multitwists be from the theorem.

For each $1\leq l\leq k$ and $n\geq 1$, $\cT_{l,n}$ is the composition of powers of Dehn twists about curves in $\sigma_l$, 
but since $r_n$ has distance at least $\ep_1$ from all completion strata except strata of multicurves having zero intersection number with 
$\alpha$, $\sigma_l$ consists of possibly the curve $\alpha$ and a number of curves disjoint from $\alpha$. 
Therefore, $\varphi_{l,n}(\alpha)=\alpha$, and $\ell_{\alpha}(\varphi_{l,n}(r_n(t)))=\ell_{\alpha}(r_n(t))$ for all $t\in [t_l,t_{l+1}]$ and all $l$.
According to Theorem \ref{thm : geodesic limit}, we have $\hat\zeta(t)\in\Teich(S)$ for all $t\in [T,T']$ except possibly the points $\{t_l\}_{l=0}^{k+1}\cap [T,T']$.
Therefore, $\ell_\alpha(\hat{\zeta}(t))>0$ for all these values of $t$.  Applying part (2) of the theorem to any such value of $t$, we have
\[\lim_{n\to\infty}\ell_{\alpha}(r_n(t))=\lim_{n\to\infty}\ell_{\alpha}(\varphi_{l,n}(r_n(t)))=\ell_\alpha(\hat{\zeta}(t))>0.\] 
This contradicts the fact that $\lim_{n\to\infty}\ell_\alpha(r_n(t))=0$ for $t \in [T,T']$.  Therefore, $\ell_{\alpha}(r(T))$ is bounded below by a positive number, depending on $T$, but independent of the ray $r$.  Thus, we have a function $f_1$, not necessarily continuous, so that $\ell_\alpha(r(T)) \geq f_1(T)$ for all $T \geq 0$.  Since $\ell_\alpha(r(t))$ is decreasing it is easy to construct a continuous function $f_1$ which also has this property.
  \end{proof}

 \begin{lem}\label{lem : tw alpha}
 There exists a function $f_2:[0,\infty)\to \mathbb R^+$ such that for any geodesic ray $r$ as above $d_{\alpha}(r(0),r(t))\leq f_2(t)$ for all $t\in[0,\infty)$.  Moreover, there exists such a function $f_2$ which is continuous.
 \end{lem}
 \begin{proof}
Suppose that such a function does not exist (not necessarily continuous). Then there is a sequence of geodesic rays $r_n$ constructed as above and a $T>0$, 
so that $\lim_{n\to\infty}d_{\alpha}(r_n(0),r_n(T))=\infty$. 
Then, since $r_n(0)=Z$ for all $n\geq 1$ we have that $\sup_t\ell_\alpha(r_n(t))\geq\inj(Z)>0$.
  Then, by Theorem \ref{thm : tw-sh} we have that $\inf_{t\in[0,T]}\ell_{\alpha} (r_n(t))\to 0$ as $n\to\infty$. 
  But this contradicts the fact that $\inf_{t\in[0,T]}\ell_{\alpha}(r_n(t))\geq \inf_{t\in[0,T]} f_1(t)>0$ for all $n\geq 1$ by Lemma \ref{lem : tw alpha}. 
  Existence of $f_2$ now follows from this contradiction.  Restricting the argument to a subinterval $[0,T'] \subset [0,T]$ we see that we can replace $f_2$ by an increasing function, and then by a continuous function, retaining the required property.
 \end{proof}

With these two lemmas in place, we now impose our final growth conditions on $\{e_k\}_k$.
Let $f_1,f_2$ be the functions from Lemmas \ref{lem : lb for alpha} and \ref{lem : tw alpha}, and for $k\in\mathbb{N}$ let 
 \begin{eqnarray}
 F_{1,k}&=&\min\{f_1(s)\mid s\in[t_k,t_{k+2}] \}\\
  F_{2,k}&=&\max\{f_2(s) \mid s\in[t_k,t_{k+2}] \}.
  \end{eqnarray}
 
As our last growth requirement for $\{e_k\}_k$, we assume $e_k$ grows fast enough that
\begin{equation}\label{eq : gr2}\lim_{k\to\infty}\frac{F_{2,k}-2\log F_{1,k}}{e_k}= 0.\end{equation}

\subsection{Limit sets}

For the remainder of the paper, we let $\{e_i^h\}_{i=0}^\infty$ be a sequence such that $e_i^h \geq K_i$, for $h = 0,1$ and all $i\in\mathbb{N}$, where $K_i$ is from Lemma~\ref{lem : close to concatenation}.  Let $\{e_k\}_k$, $\{\gamma_k\}_k$, $\{t_k\}_k$, $\{t'_k\}_k$, $\{X_k\}_k$ be as in Notation~\ref{not : seq k}, and assume that $\{e_k\}_k$ satisfies (\ref{eq : gr1}) and (\ref{eq : gr2}).  The following immediately implies Theorem~\ref{thm:main}.


\begin{thm}\label{thm : limit of r}
The limit set of $r$ in the Thurston compactification of $\Teich(S)$ is the $1$--simplex $[[\bar\lambda_0],[\bar\lambda_1]]$ of projective classes of measures supported on $\lambda_0 \cup \lambda_1$.
\end{thm}

For curves $\delta,\gamma\in \cC_0(S)$ and any time $s\in[0,\infty)$, as in (\ref{Eq:contribution delta gamma X}), let
\begin{equation}\label{eq : ld(g,t)}
\ell_\delta(\gamma,s)= \ell_\delta(\gamma,r(s)) = \I(\delta,\gamma)\Big(w_{\gamma}(r(s))+\ell_\gamma(r(s))\tw_\gamma(\delta,r(s))\Big).
\end{equation}

Now suppose that $\{s_k\}_k$ is a sequence such that $s_k \in [t_k',t_{k+1}']$.  Pass to a subsequence $\{s_k\}_{k \in \mathcal K}$ so that $r(s_k) \to [\bar \nu]$ in the Thurston compactification (to avoid cluttering the notation with additional subscripts, we have chosen to index a subsequence using a subset $\mathcal K \subset \mathbb N$).  Let $\{u_k\}_{k\in\mathcal{K}}$ be a scaling sequence, so that 
\[ \lim_{k \to \infty} u_k \ell_{\delta}(r(s_k)) = i(\delta,\bar \nu),\]
for all curves $\delta$.

By Corollary~\ref{cor : general times r} and Proposition~\ref{prop : properties of r}, the curves $\gamma_k,\gamma_{k+1},\alpha$ form a uniformly bounded length pants decomposition on $r(s_k)$.  Consequently, by Theorem \ref{thm:combinatorial length} we
 obtain the following expansion for the length of the curve $\delta\in\mathcal C_0(S)$ at $r(s_k)$,
\begin{eqnarray}\label{eq : length of d}
\ell_{\delta}(r(s_k))&=& \ell_{\delta}(\gamma_k,s_k)+\ell_{\delta}(\gamma_{k+1},s_k) +\ell_{\delta}(\alpha,s_k)\nonumber\\
 & & +O \left( \I(\delta,\gamma_k) \right) + O \left( \I(\delta,\gamma_{k+1}) \right) + O \left( \I(\delta,\alpha) \right)
\end{eqnarray}
where the constant of the $O$ notation depends only on the uniform upper bounds for the lengths of $\gamma_k$, $\gamma_{k+1}$, and $\alpha$.

The next proposition shows that only two of the terms in (\ref{eq : length of d}) are actually relevant.
\begin{prop} \label{prop : relevant terms} With notation as above, and $\delta \in \cC_0(S)$ with $i(\delta,\alpha) \neq 0$, we have
\[ \I(\delta,\bar \nu) = \lim_{k \to \infty} u_k \ell_\delta(r(s_k)) = \lim_{k \to \infty} u_k(\ell_{\delta}(\gamma_k,s_k) + \ell_{\delta}(\gamma_{k+1},s_k)).\]
\end{prop}
For this, we will need the following lemma.
\begin{lem} \label{lem : k contribution 1} With notation as above,
\[ w_{\gamma_k}(r(s_k)) + \ell_{\gamma_k}(r(s_k)) \tw_{\gamma_k}(\delta,r(s_k)) \asya \ell_{\gamma_k}(r(s_k)) e_k,\]
where the constant in the coarse equation depends on $\delta$, but not on $k$.
\end{lem}
\begin{proof} By Corollary~\ref{cor : general times r},
\begin{center}$ \ell_{\gamma_k}(r(s_k))\leq C'$ and $\ell_{\gamma_{k+2}}(r(s_k)) \leq C'$. \end{center}
Then since $\I(\gamma_k,\gamma_{k+2}) = 1$, $\ell_{\gamma_k}(r(s_k))$ is also uniformly bounded below,  for otherwise, by Lemma~\ref{L:collar}, $\ell_{\gamma_{k+2}}(r(s_k))$ would be  unbounded. So we have  $\ell_{\gamma_k}(r(s_k)) \asym 1$ and hence again by Lemma~\ref{L:collar} we have $ w_{\gamma_k}(r(s_k)) \asym 1$.
Moreover, by Lemma~\ref{lem : est tw gamma} we have $\tw_{\gamma_k}(\delta,r(s_k)) \asya e_k$, and so the lemma follows.
\end{proof}

\begin{proof}[Proof of Proposition~\ref{prop : relevant terms}.] First, observe that by Corollary~\ref{cor : rel small intersect} (and since $\alpha$ is a fixed curve and $e_k \to \infty$), we have
\begin{equation} \label{eq : big O terms} \lim_{k \to \infty} \frac{\I(\delta,\gamma_k)}{e_k} =0,\; \lim_{k \to \infty} \frac{\I(\delta,\gamma_{k+1})}{e_k} =0, \;\text{and}\; \lim_{k \to \infty}\frac{\I(\delta,\alpha)}{e_k} = 0.
\end{equation}  
As in the proof of Lemma~\ref{lem : k contribution 1}, $\ell_{\gamma_k}(r(s_k)) \asym 1$, and so $\ell_{\gamma_k}(r(s_k)) e_k \to \infty$ and $\I(\delta,\gamma_k) \to \infty$.  From Lemma~\ref{lem : k contribution 1} and (\ref{eq : ld(g,t)}), we have $\ell_{\delta}(\gamma_k,s_k) \gmul e_k$.  
Moreover, $u_k \ell_\delta(r(s_k)) \to \I(\delta,\bar \nu) > 0$, then by (\ref{eq : length of d}), $u_k \lmul \tfrac1{e_k}$. Combining this with (\ref{eq : big O terms}) and appealing to (\ref{eq : length of d}) again, we see that
\[ i(\delta,\bar \nu) = \lim_{k \to \infty} u_k \ell_\delta(r(s_k)) = \lim_{k \to \infty} u_k(\ell_{\delta}(\gamma_k,s_k) + \ell_{\delta}(\gamma_{k+1},s_k) + \ell_{\delta}(\alpha,s_k)).\]

By similar reasoning, to eliminate the last term (and thus prove the proposition), it suffices to prove
\begin{equation} \label{eq : eliminate alpha} \lim_{k \to \infty} \frac{\ell_{\delta}(\alpha,s_k)}{e_k} = 0.
\end{equation}

To do this, first note that by Lemma~\ref{lem : lb for alpha}, $\ell_\alpha(r(s_k))\geq f_1(s_k) \geq F_{1,k}$, and so by Lemma~\ref{L:collar} we have
\[ w_{\alpha}(r(s_k)) \asya -2\log (\ell_\alpha(r(s_k))) \leq -2 \log(F_{1,k}). \]
By Lemma~\ref{lem : tw alpha}, we also have
\[ \tw_{\alpha}(\delta,r(s_k)) \asya d_\alpha(r(0),r(s_k)) \leq f_2(s_k) \leq F_{2,k},\]
where the additive constant depends on $\delta$.
Therefore, since $F_{1,k} \ladd 1$, and since $\ell_{\alpha}(r(s_k))$ is uniformly bounded by Proposition~\ref{prop : properties of r}, we have
\[ \ell_{\delta}(\alpha,s_k) \ladd \I(\delta,\alpha) \Big(-2 \log(F_{1,k})+ F_{2,k} \Big), \] 
with additive error that again depends on $\delta$.
By our growth condition (\ref{eq : gr2}), since $\I(\delta,\alpha)$ does not depend on $k$, there is a constant $c' > 0$ so that 
\[ \lim_{k \to \infty} \frac{\ell_{\delta}(\alpha,s_k)}{e_k} \leq  \lim_{k \to \infty} \frac{\I(\delta,\alpha) \Big(-2 \log(F_{1,k})+ F_{2,k} \Big)+c'}{e_k} = 0.\]
This proves (\ref{eq : eliminate alpha}), and hence the proposition.
\end{proof}

Continue to let $\{s_k\}_{k \in \mathcal K}$ be a sequence with $s_k \in [t_k',t_{k+1}']$ as above, and suppose that $r(s_k) \to [\bar \nu]$ as $k\to\infty$ in the Thurston compactification and that $\{u_k\}_k$ is a scaling sequence.  We set
\[ x(s_k) = w_{\gamma_k}(r(s_k)) + \ell_{\gamma_k}(r(s_k)) \tw_{\gamma_k}(\gamma_{\bar k},r(s_k)),\]
and
\[ y(s_k) = w_{\gamma_{k+1}}(r(s_k)) + \ell_{\gamma_{k+1}}(r(s_k)) \tw_{\gamma_{k+1}}(\gamma_{\overline{k+1}},r(s_k)).\]
\begin{lem} \label{lem : estimating coeffs} For any $\delta \in \cC_0(S)$ with $\I(\delta,\alpha) \neq 0$, we have
\[ \lim_{k \to \infty} \frac{x(s_k) \I(\delta,\gamma_k) + y(s_k) \I(\delta,\gamma_{k+1})}{\ell_{\gamma_k}(\delta,s_k) + \ell_{\gamma_{k+1}}(\delta,s_k)} = 1.\]
For $s_k = t_k'$, we have
\[ \lim_{k \to \infty} \frac{x(t_k') \I(\delta,\gamma_k)}{\ell_{\gamma_k}(\delta,t_k') + \ell_{\gamma_{k+1}}(\delta,t_k')} = 1.\]
\end{lem}
\begin{proof}  As in the proof of Lemma~\ref{lem : est tw gamma}
\[ \tw_{\gamma_k}(\delta,r(s_k)) \asya \tw_{\gamma_k}(\gamma_{\bar k},r(s_k))\]
and
\[ \tw_{\gamma_{k+1}}(\delta,r(s_k)) \asya \tw_{\gamma_k}(\gamma_{\overline{k+1}},r(s_k)) \]
where the implicit constant in these coarse equations depends on $\delta$.

According to Corollary~\ref{cor : general times r}, $\ell_{\gamma_k}(r(s_k)) \leq C'$.  From the preceding coarse equations and Lemma~\ref{lem : k contribution 1}, we have
\begin{equation}\label{eq : x estimate}
 x(s_k) \asya w_{\gamma_k}(r(s_k)) + \ell_{\gamma_k}(r(s_k)) \tw_{\gamma_k}(\delta,r(s_k))  \asya \ell_{\gamma_k}(r(s_k)) e_k.
  \end{equation}
Since $e_k \to \infty$ as $k \to \infty$, the following is immediate:
\begin{equation} \label{eq : asymptotic x} \lim_{k \to \infty} \frac{x(s_k) \I(\delta,\gamma_k)}{\ell_\delta(\gamma_k,s_k)} = \lim_{k \to \infty} \frac{x(s_k)}{w_{\gamma_k}(r(s_k)) + \ell_{\gamma_k}(r(s_k)) \tw_{\gamma_k}(\delta,r(s_k)) } = 1.
\end{equation}


Similar to (\ref{eq : x estimate}) we have 
\begin{equation} \label{eq : y estimate} 
y(s_k) \asya w_{\gamma_{k+1}}(r(s_k)) + \ell_{\gamma_{k+1}}(r(s_k)) \tw_{\gamma_{k+1}}(\delta,r(s_k)) = \tfrac{\ell_{\gamma_{k+1}}(\delta,s_k)}{\I(\delta,\gamma_{k+1})}. 
\end{equation}
By (\ref{eqn:intersection index change}) and the growth condition (\ref{eq : gr1}) we have
\[ \lim_{k \to \infty} \frac{\I(\delta,\gamma_{k+1})}{e_k} = 0.\]
After passing to a subsequence, there are two cases to consider:

\medskip
\noindent{\bf Case 1.} There exists $R >0$ so that $y(s_k) \leq R$ for all $k$.\\

In this case, appealing to (\ref{eq : big O terms}) and (\ref{eq : y estimate}) we have
\[ 0 = \lim_{k \to \infty} \frac{y(s_k) \I(\delta,\gamma_{k+1})}{e_k} = \lim_{k \to \infty} \frac{\ell_{\gamma_{k+1}}(\delta,s_k)}{e_k},\]
and thus
\begin{eqnarray*}
\lim_{k \to \infty} \frac{x(s_k) \I(\delta,\gamma_k) + y(s_k) \I(\delta,\gamma_{k+1})}{\ell_{\gamma_k}(\delta,s_k) + \ell_{\gamma_{k+1}}(\delta,s_k)} & = &  \lim_{k \to \infty} \frac{\tfrac{x(s_k) \I(\delta,\gamma_k)}{e_k} + \tfrac{y(s_k) \I(\delta,\gamma_{k+1})}{e_k}}{\tfrac{\ell_{\gamma_k}(\delta,s_k)}{e_k} + \tfrac{\ell_{\gamma_{k+1}}(\delta,s_k)}{e_k}}\\ 
& = & \lim_{k \to \infty} \frac{x(s_k) \I(\delta,\gamma_k)}{\ell_{\gamma_k}(\delta,s_k)} = 1.
\end{eqnarray*}

\noindent{\bf Case 2.} $\displaystyle{\lim_{k \to \infty} y(s_k) = \infty}$.\\

Here, we can argue as for $x(s_k)$, appealing to (\ref{eq : y estimate}) to deduce that
\[ \lim_{k \to \infty} \frac{y(s_k) \I(\delta,\gamma_{k+1})}{\ell_{\gamma_{k+1}}(\delta,s_k)} = 1.\]
Combined with (\ref{eq : asymptotic x}) we have
\[ \lim_{k \to \infty} \frac{x(s_k) \I(\delta,\gamma_k) + y(s_k) \I(\delta,\gamma_{k+1})}{\ell_{\gamma_k}(\delta,s_k) + \ell_{\gamma_{k+1}}(\delta,s_k)} = 1.\]

These two cases prove the first claim of the lemma.  
For the second claim, when $s_k = t_k'$, we note that by Corollary~\ref{cor : special times r} we have
\[ \ell_{\gamma_{k-1}}(r(t_k')), \ell_{\gamma_{k+1}}(r(t_k')) \leq C',\]
and so by Lemma~\ref{L:collar} (as in the proof of Lemma~\ref{lem : k contribution 1}) we have
\[ w_{\gamma_{k+1}}(r(t_k')) \asym 1 \quad \mbox{ and } \quad \ell_{\gamma_{k+1}}(r(t_k')) \asym 1 \]
so since $\tw_{\gamma_{k+1}}(\delta,r(t_k')) \asya 1$ by Lemma~\ref{lem : est tw gamma}, it follows that $y(t_k')$ is uniformly bounded, and thus as in Case 1, we deduce 
\[ \lim_{k \to \infty} \frac{x(t_k') \I(\delta,\gamma_k)}{\ell_{\gamma_k}(\delta,t_k') + \ell_{\gamma_{k+1}}(\delta,t_k')} = 1,\]
completing the proof.
\end{proof}

We are now ready for the
\begin{proof}[Proof of Theorem \ref{thm : limit of r}]   First, we show that $[\bar \lambda_0]$ and $[\bar \lambda_1]$ are in the limit set $\Lambda$ of $r$.  
Consider the sequence of times $\{t_{2k}'\}$ and pass to a subsequence so that $r(t_{2k}') \to [\bar \nu]$ in the Thurston compactification and let $\{u_k\}$ be a scaling sequence for $r(t_{2k}')$.
Let $\delta$ be any curve with $\I(\delta,\bar \nu) \neq 0$ and $\I(\delta,\alpha) \neq 0$.
By the second part of Lemma~\ref{lem : estimating coeffs}, together with Proposition~\ref{prop : relevant terms} we have
\begin{eqnarray*} 1 & = & \lim_{k \to \infty} \frac{x(t_{2k}') \I(\delta,\gamma_{2k})}{\ell_{\gamma_{2k}}(\delta,t_{2k}') + \ell_{\gamma_{2k+1}}(\delta,t_{2k}')}\\
&=&  \lim_{k \to \infty} \frac{u_k x(t_{2k}') \I(\delta,\gamma_{2k})}{u_k(\ell_{\gamma_{2k}}(\delta,t_{2k}') + \ell_{\gamma_{2k+1}}(\delta,t_{2k}'))}\\
& = & \frac{\displaystyle{\lim_{k \to \infty} \I(\delta,u_k x(t_{2k}') \gamma_{2k})}}{\I(\delta,\bar \nu)}.
\end{eqnarray*}
Therefore $\displaystyle{\lim_{k \to \infty} \I(\delta, u_k x(t_{2k}') \gamma_{2k}) = \I(\delta,\bar \nu)}$.
We apply this to a set of curves $\delta_1,\ldots, \delta_N$ sufficient for determining a measured lamination (see $\S$\ref{sec:prelim}), 
and so deduce that $\displaystyle{\lim_{k \to \infty} u_k x(t_{2k}') \gamma_{2k} = \bar \nu}$.

On the other hand, $[\gamma_{2k}] \to [\bar \lambda_0]$, hence $[\bar \nu] = [\bar \lambda_0]$, and so $[\bar \lambda_0]$ is in $\Lambda$.  
A similar argument using the sequence $\{t_{2k+1}'\}$ shows that $[\bar \lambda_1] \in \Lambda$.

Now suppose that $\{s_k\}_k$ is an arbitrary sequence so that $r(s_k) \to [\bar \nu]$ and let $\{u_k\}_k$ be a scaling sequence.
Adjusting indices and passing to a subsequence we can assume that $s_k \in [t_k',t_{k+1}']$ for all $k \in \mathcal K$ (some subset $\mathcal K \subseteq \mathbb N$).  
Passing to a further subsequence, if necessary, we may assume that $\mathcal{K}$ is either a subsequence of even integers or odd integers.
Arguing as above, appealing to the first part of Lemma~\ref{lem : estimating coeffs} and Proposition~\ref{prop : relevant terms} we have
\begin{eqnarray*} 
1 & = & \lim_{k \to \infty} \frac{x(s_k) \I(\delta,\gamma_k) + y(s_k)\I(\delta,\gamma_{k+1})}{\ell_{\gamma_k}(\delta,s_k) + \ell_{\gamma_{k+1}}(\delta,s_k)} \\
& = & \lim_{k \to \infty} \frac{u_k x(s_k) \I(\delta,\gamma_k) + u_k y(s_k)\I(\delta,\gamma_{k+1})}{u_k(\ell_{\gamma_k}(\delta,s_k) + \ell_{\gamma_{k+1}}(\delta,s_k))}\\
& = & \frac{\displaystyle{\lim_{k \to \infty} \I\big(\delta,u_k (x(s_k) \gamma_k + y(s_k) \gamma_{k+1})\big)}}{\I(\delta,\bar \nu)}.
\end{eqnarray*}
So, $\displaystyle{\lim_{k \to \infty} u_k (x(s_k) \gamma_k + y(s_k) \gamma_{k+1}) = \I(\delta,\bar \nu)}$, and as above
\[ \bar \nu = \lim_{k \to \infty} u_k(x(s_k) \gamma_k + y(s_k) \gamma_{k+1}) = \lim_{k \to \infty} u_kx(s_k) \gamma_k + \lim_{k \to \infty} u_k y(s_k) \gamma_{k+1}.\] 
Now, if $\mathcal{K}$ is a subset of even integers, then since the projective classes of the curves with even indices converge to $[\bar\lambda_0]$, 
the first limit on the right hand-side above is a multiple of $\bar\lambda_0$, and since the projective classes of the curves with odd indices converge to $[\bar\lambda_1]$
 the second limit above is a multiple of $\bar\lambda_1$, and hence
$[\bar \nu] \in [[\bar \lambda_0],[\bar \lambda_1]]$. When $\mathcal{K}$ is a subset of odd integers we have a similar conclusion.
This implies that $\Lambda$ is contained in $[[\bar \lambda_0],[\bar \lambda_1]]$.  Since $\Lambda$ contains the endpoints and is connected, it is the entire $1$--simplex, as was desired.
\end{proof}
\bibliographystyle{amsalpha}
\bibliography{reference}

\providecommand{\bysame}{\leavevmode\hbox to3em{\hrulefill}\thinspace}
\providecommand{\MR}{\relax\ifhmode\unskip\space\fi MR }
\providecommand{\MRhref}[2]{%
  \href{http://www.ams.org/mathscinet-getitem?mr=#1}{#2}
}
\providecommand{\href}[2]{#2}
\begin{thebibliography}{BLMR16b}

\bibitem[BLMR16a]{nue2}
Jeffrey Brock, Christopher Leininger, Babak Modami, and Kasra Rafi, \emph{Limit
  sets of {T}eicm\"{u}ller geodesics with minimal non-uniquely vertical
  laminations, {II}}, {J.} {R}eine {A}ngew. {M}ath. to appear.
  ar{X}iv:1601.03368 (2016).

\bibitem[BLMR16b]{wplimit}
\bysame, \emph{Limit sets of {W}eil-{P}etersson geodesics}, Int. Math. Res.
  Not. (IMRN) to appear, ar{X}iv:1611.02197 (2016).

\bibitem[BM07]{brockmargalit}
Jeffrey Brock and Dan Margalit, \emph{Weil-{P}etersson isometries via the pants
  complex}, Proc. Amer. Math. Soc. \textbf{135} (2007), no.~3, 795--803.

\bibitem[BM15]{wprecurnue}
Jeffrey Brock and Babak Modami, \emph{Recurrent {W}eil-{P}etersson geodesic
  rays with non-uniquely ergodic ending laminations}, Geom. Topol. \textbf{19}
  (2015), no.~6, 3565--3601.

\bibitem[BMM10]{bmm1}
Jeffrey Brock, Howard Masur, and Yair Minsky, \emph{Asymptotics of
  {W}eil-{P}etersson geodesics. {I}. {E}nding laminations, recurrence, and
  flows}, Geom. Funct. Anal. \textbf{19} (2010), no.~5, 1229--1257.

\bibitem[BMM11]{bmm2}
\bysame, \emph{Asymptotics of {W}eil-{P}etersson geodesics {II}: bounded
  geometry and unbounded entropy}, Geom. Funct. Anal. \textbf{21} (2011),
  no.~4, 820--850.

\bibitem[Bro03]{br}
Jeffrey~F. Brock, \emph{The {W}eil-{P}etersson metric and volumes of
  3-dimensional hyperbolic convex cores}, J. Amer. Math. Soc. \textbf{16}
  (2003), no.~3, 495--535.

\bibitem[Bus10]{buser}
Peter Buser, \emph{Geometry and spectra of compact {R}iemann surfaces}, Modern
  Birkh\"auser Classics, Birkh\"auser Boston Inc., Boston, MA, 2010, Reprint of
  the 1992 edition.

\bibitem[CMW14]{nuechaikaetl}
Jon Chaika, Howard Masur, and Michael Wolf, \emph{Limits in {PMF} of
  {T}eichm\"{u}ller geodesics}, J. Reine Angew. Math. to appear,
  arXiv:1406.0564 (2014).

\bibitem[CRS08]{lineminimateichgeod}
Young-Eun Choi, Kasra Rafi, and Caroline Series, \emph{Lines of minima and
  {T}eichm\"uller geodesics}, Geom. Funct. Anal. \textbf{18} (2008), no.~3,
  698--754.

\bibitem[DW03]{dwwp}
Georgios Daskalopoulos and Richard Wentworth, \emph{Classification of
  {W}eil-{P}etersson isometries}, Amer. J. Math. \textbf{125} (2003), no.~4,
  941--975.

\bibitem[FLP79]{FLP}
A.~Fathi, F.~Laudenbach, and V.~Poenaru, \emph{Travaux de {T}hurston sur les
  surfaces}, Ast{\'e}risque No. 66-67 (1979), 1--286.

\bibitem[Ham15]{Hamen-teich-wp}
Ursula Hamenst{\"a}dt, \emph{{W}eil-{P}etersson flow and {T}eichm\"{u}ller
  flow}, arXiv:1505.01113 (2015).

\bibitem[Ker80]{Kerck-asymp-teich}
Steven~P. Kerckhoff, \emph{The asymptotic geometry of {T}eichm\"uller space},
  Topology \textbf{19} (1980), no.~1, 23--41.

\bibitem[Khi64]{Khinchin-contfrac}
A.~Ya. Khinchin, \emph{Continued fractions}, The University of Chicago Press,
  Chicago, Ill.-London, 1964.

\bibitem[Len08]{lenzhen}
Anna Lenzhen, \emph{Teichm\"uller geodesics that do not have a limit in
  {${\mathcal{PMF}}$}}, Geom. Topol. \textbf{12} (2008), no.~1, 177--197.

\bibitem[LLR13]{nonuniqueerg}
Christopher Leininger, Anna Lenzhen, and Kasra Rafi, \emph{Limit sets of
  {T}eichm\"{u}ller geodesics with minimal non-uniquely ergodic vertical
  foliation}, {J.} {R}eine {A}ngew. {M}ath. to appear. arXiv: 1312.2305 (2013).

\bibitem[LM10]{lenmasdivergence}
Anna Lenzhen and Howard Masur, \emph{Criteria for the divergence of pairs of
  {T}eichm\"uller geodesics}, Geom. Dedicata \textbf{144} (2010), 191--210.

\bibitem[LMR16]{2dlimit}
Anna Lenzhen, Babak Modami, and Kasra Rafi, \emph{Teichmueller geodesics with
  d-dimensional limit sets}, {J.} {M}od {D}yn. to appear. arXiv:1608.07945
  (2016).

\bibitem[LRT15]{Shadow}
Anna Lenzhen, Kasra Rafi, and Jing Tao, \emph{The shadow of a {T}hurston
  geodesic to the curve graph}, J. Topol. \textbf{8} (2015), no.~4, 1085--1118.

\bibitem[Mas76]{maswp}
Howard Masur, \emph{Extension of the {W}eil-{P}etersson metric to the boundary
  of {T}eichmuller space}, Duke Math. J. \textbf{43} (1976), no.~3, 623--635.

\bibitem[Mas82]{2bdriesteich}
\bysame, \emph{Two boundaries of {T}eichm\"uller space}, Duke Math. J.
  \textbf{49} (1982), no.~1, 183--190.

\bibitem[Mas85]{maskit}
Bernard Maskit, \emph{Comparison of hyperbolic and extremal lengths}, Ann.
  Acad. Sci. Fenn. Ser. A I Math. \textbf{10} (1985), 381--386. \MR{802500}

\bibitem[Min96]{geometricccplx}
Yair~N. Minsky, \emph{A geometric approach to the complex of curves on a
  surface}, Topology and {T}eichm\"uller spaces ({K}atinkulta, 1995), World
  Sci. Publ., River Edge, NJ, 1996, pp.~149--158. \MR{1659683}

\bibitem[MM99]{mm1}
Howard~A. Masur and Yair~N. Minsky, \emph{Geometry of the complex of curves.
  {I}. {H}yperbolicity}, Invent. Math. \textbf{138} (1999), no.~1, 103--149.

\bibitem[MM00]{mm2}
H.~A. Masur and Y.~N. Minsky, \emph{Geometry of the complex of curves. {II}.
  {H}ierarchical structure}, Geom. Funct. Anal. \textbf{10} (2000), no.~4,
  902--974.

\bibitem[Mod15]{wpbehavior}
Babak Modami, \emph{Prescribing the behavior of {W}eil--{P}etersson geodesics
  in the moduli space of {R}iemann surfaces}, J. Topol. Anal. \textbf{7}
  (2015), no.~4, 543--676.

\bibitem[Mod16]{asympdiv}
\bysame, \emph{Asymptotics of a class of {W}eil--{P}etersson geodesics and
  divergence of {W}eil--{P}etersson geodesics}, Algebr. Geom. Topol.
  \textbf{16} (2016), no.~1, 267--323.

\bibitem[Ser85]{modsurf}
Caroline Series, \emph{The modular surface and continued fractions}, J. London
  Math. Soc. (2) \textbf{31} (1985), no.~1, 69--80.

\bibitem[Wol03]{wols}
Scott~A. Wolpert, \emph{Geometry of the {W}eil-{P}etersson completion of
  {T}eichm\"uller space}, Surveys in differential geometry, {V}ol.\ {VIII}
  ({B}oston, {MA}, 2002), Surv. Differ. Geom., VIII, Int. Press, Somerville,
  MA, 2003, pp.~357--393.

\bibitem[Wol08]{wolb}
\bysame, \emph{Behavior of geodesic-length functions on {T}eichm\"uller space},
  J. Differential Geom. \textbf{79} (2008), no.~2, 277--334.

\bibitem[Wol10]{wol}
\bysame, \emph{Families of {R}iemann surfaces and {W}eil-{P}etersson geometry},
  CBMS Regional Conference Series in Mathematics, vol. 113, Published for the
  Conference Board of the Mathematical Sciences, Washington, DC, 2010.

\end{thebibliography}
\end{document}